\documentclass[a4paper,11pt]{amsart}

\usepackage{amsmath, amsthm, amsfonts, amssymb, enumerate}
\usepackage[bookmarks=false]{hyperref} 
\usepackage{xcolor}
\usepackage{verbatim}
\usepackage[normalem]{ulem}
\usepackage{tikz-cd}
\usepackage[mathscr]{eucal}

\title[Relative solidity in measure equivalence]{Relative solidity for biexact groups in measure equivalence}
\author{Changying Ding}
\address{Department of Mathematics, UCLA, Los Angeles, CA 90095, USA}
\email{cding@math.ucla.edu}

\author{Daniel Drimbe}
\address{Department of Mathematics, The University of Iowa, Iowa City, IA 52242, USA}
\email{daniel-drimbe@uiowa.edu}

\newtheorem{main}{Theorem}
\newtheorem{mcor}[main]{Corollary}

\newtheorem{conjecture}[main]{Conjecture}

\newtheorem{thm}{Theorem}[section]
\newtheorem{prop}[thm]{Proposition}

\newtheorem{cor}[thm]{Corollary}
\newtheorem{lem}[thm]{Lemma}
\theoremstyle{definition}
\newtheorem{defn}[thm]{Definition}
\newtheorem{defn/lem}[thm]{Definition/Lemma}
\newtheorem{rem}[thm]{Remark}
\newtheorem{nota}[thm]{Notation}
\newtheorem{examp}[thm]{Example}

\newtheorem{note}[thm]{Notation}

\newcommand{\B}{{\mathbb B}}
\newcommand{\C}{{\mathbb C}}
\newcommand{\F}{{\mathbb F}}

\newcommand{\K}{{\mathbb K}}
\newcommand{\M}{{\mathbb M}}
\newcommand{\N}{{\mathbb N}}

\newcommand{\bS}{{\mathbb S}}

\newcommand{\X}{{\mathbb X}}
\newcommand{\Z}{{\mathbb Z}}
\newcommand{\cP}{{\mathcal P}}
\newcommand{\cR}{{\mathcal R}}
\newcommand{\cC}{{\mathcal C}}
\newcommand{\cB}{{\mathcal B}}
\newcommand{\cF}{{\mathcal F}}
\newcommand{\cG}{{\mathcal G}}

\newcommand{\cK}{{\mathcal K}}
\newcommand{\cL}{{\mathcal L}}
\newcommand{\cM}{{\mathcal M}}
\newcommand{\cN}{{\mathcal N}}

\newcommand{\cS}{{\mathcal S}}
\newcommand{\cU}{{\mathcal U}}

\newcommand{\cZ}{{\mathcal Z}}

\newcommand{\car}{\curvearrowright}

\newcommand{\Ad}{\operatorname{Ad}}

\newcommand{\Aut}{\operatorname{Aut}}

\newcommand{\id}{\operatorname{id}}

\newcommand{\SL}{\operatorname{SL}}

\newcommand{\Tr}{\operatorname{Tr}}

\newcommand{\ot}{\otimes}
\newcommand{\ovt}{\, \bar{\otimes}\,}

\newcommand{\alg}{{\operatorname{alg}}}

\newcommand{\ds}{{\sharp\kern-.5pt\sharp}}
\newcommand{\op}{{\rm op}}

\newcommand{\actson}{{\, \curvearrowright \,}}

\begin{document}

\maketitle
\begin{abstract}
We demonstrate a relative solidity property for the product of a nonamenable biexact group with an arbitrary infinite group in the measure equivalence setting.
Among other applications, we obtain the following unique product decomposition for products of nonamenable biexact groups, strengthening \cite{Sa09}:
	for any nonamenable biexact groups $\Gamma_1,\cdots, \Gamma_n$,
	if a product group $\Lambda_1\times \Lambda_2$ is measure equivalent to $\times_{k=1}^n\Gamma_k$, then there exists a partition $T_1\sqcup T_2=\{1,\dots, n\}$ such that $\Lambda_i$ is measure equivalent to $\times_{k\in T_i}\Gamma_k$ for $i=1,2$.
\end{abstract}

\section{Introduction}

Two countable groups $\Gamma$ and $\Lambda$ are said to be {\it measure equivalent} in the sense of Gromov \cite{Gr93} if there exist commuting free measure preserving actions of $\Gamma$ and $\Lambda$ on a standard $\sigma$-finite measure space $(\Omega,\mu)$ such that both the actions  $\Gamma\car (\Omega,\mu)$ and $\Lambda \car (\Omega,\mu)$ admit a finite measure fundamental domain. The classification of countable groups up to measure equivalence is a central theme in measured group theory and many spectacular innovations have been made in the last 25 years, see for instance the introduction of \cite{HHI21}. 

In their seminal work \cite{MS02}, Monod and Shalom used techniques from bounded cohomology theory to obtain the following general unique prime factorization result: if $\times_{k=1}^n \Gamma_k$ is a product of non-elementary torsion-free hyperbolic groups %(more generally, of groups belonging to class $\mathcal C_{\rm reg}$, see \cite[Notation 1.2]{MS02})
that  is measure equivalent to a product $\times_{i=1}^m \Lambda_i$ of torsion-free groups and $m\ge n$, then $n = m$ and after a permutation of the indices, $\Gamma_i$ is measure equivalent to $\Lambda_i$ for any $1\leq i\leq n$.
Sako then used  C$^*$-algebraic techniques from \cite{Oz04,OP03}, to extend in \cite{Sa09} (see also \cite{ChSi13}) the above unique prime factorization results to products of nonamenable biexact groups. For additional such unique prime factorization results, see \cite{DHI16,Dri23}.

In the framework of II$_1$ factors, Ozawa and Popa obtained in \cite{OP03} a striking analogue to Monod and Shalom's unique prime factorization result. They showed that if $\Gamma_1,\dots, \Gamma_n$ are i.c.c. (infinite conjugacy class)  nonamenable hyperbolic groups (more generally, biexact groups) such that the group von Neumann algebra $L(\times_{k=1}^n \Gamma_k)$ is stably isomorphic to a tensor product of  ${\rm II}_1$ factors  $\ovt_{i=1}^m N_i$  and $m\ge n$, then $m=n$ and after a permutation of the indices, $L(\Gamma_i)$ is stably isomorphic to $L(\Lambda_i)$ for any $1\leq i\leq n$. In fact, Ozawa and Popa proved a more general unique prime factorization phenomenon by classifying {\it all tensor product decompositions} of $L(\times_{k=1}^n \Gamma_k)$. More precisely, they showed that if $L(\times_{k=1}^n \Gamma_k)$ is stably isomorphic to a tensor product of ${\rm II}_1$ factors $N_1\ovt N_2$, 
	then there exists a partition $T_1\sqcup T_2=\{1,\dots, n\}$ such that $N_i$ is stably isomorphic to $L(\times_{k\in T_i}\Gamma_k)$ for any $1\leq i\leq 2$.

%Nevertheless, a subtle difference can be observed between the unique prime factorization of Monod and Shalom \cite{MS02} and the one of Ozawa and Popa \cite{OP03}. To explain this, let $\Gamma_1,\dots,\Gamma_n$ be hyperbolic groups. The result of \cite{MS02} allows to classify groups $\Lambda_1\times\dots\times\Lambda_m$ that are measure equivalent to $\Gamma_1\times\dots\times\Gamma_n$ as long as $m\ge n$. On the other hand, for \cite{OP03} the condition $m\ge n$ is not needed since Ozawa and Popa are able to completely classify all tensor product decompositions of $L(\times_{i=1}^n \Gamma_k)$.

In our first main result we show that a stronger form of Sako's unique prime factorization result \cite{Sa09} holds 
 by classifying {\it all i.c.c.\ product groups} that are measure equivalent to a product of nonamenable biexact groups.

\begin{main}\label{theorem.upf.biexact}
Let $\Gamma_1, \dots, \Gamma_n$ be  nonamenable biexact groups. 
Suppose $\times_{k=1}^n\Gamma_k$ is measure equivalent to a product $\Lambda_1\times \Lambda_2$ of i.c.c. groups.
Then there exists a partition $T_1\sqcup T_2=\{1,\dots, n\}$
	such that $\Lambda_j$ is measure equivalent to $\times_{i\in T_j} \Gamma_i$ for $j=1,2$.
\end{main}

Note that Theorem~\ref{theorem.upf.biexact} is a consequence of \cite{DHI16} under the stronger assumption that $\Gamma_1,\dots,\Gamma_n$ are non-elementary hyperbolic groups. Results in \cite{DHI16} were achieved by using a combination of tools from Popa's deformation/rigidity theory including
the fundamental work of Popa and Vaes \cite{PoVa14II} 
which shows that any  hyperbolic (more generally, groups that are weakly amenable and biexact)  satisfies the following {\it relative strong solidity property}: if $\Gamma\car N$ is a trace preserving action of a hyperbolic group and $P\subset M:=N\rtimes\Gamma$ is a von Neumann subalgebra that is amenable relative to $N$ inside $M$, then $P\prec_M N$ or $\mathcal {N}_{M}(P)''$ is amenable relative to $N$. An essential ingredient used in \cite{DHI16} is that any group $\Gamma$ with the relative strong solidity property satisfies the following {\it relative solidity property}: %In fact, the following consequence of \cite{PoVa14II} is essential in \cite{DHI16}: 
    if $\Gamma\car N$ is a trace preserving action of a hyperbolic group and $P,Q\subset M:=N\rtimes\Gamma$ are commuting von Neumann subalgebras, then $P\prec_{M} N$ or $Q$ is amenable relative to $N$ inside $M$ (see \cite[Lemma 5.2]{KV15}). Note that the converse is false: there exist groups (e.g. biexact group that are not weakly amenable such as $\mathbb Z\wr\mathbb F_2$) that satisfies the relative solidity property, but not the strong relative property.
    Here, the intertwining is in the sense of Popa \cite{Po06A} and relative amenability is in the sense of Ozawa and Popa \cite{OzPo10I}.

% \textcolor{cyan}{For perspective, the above is obtained in \cite[Theorem C]{DHI16} under the extra assumption that each $\Gamma_k$ is weakly amenable
% 	by building upon the fundamental work of Popa-Vaes \cite{PoVa14II} on relative strong solidity
% 	for groups that are both weakly amenable and biexact.\\
% In contrast, central to our approach to Theorem~\ref{theorem.upf.biexact} is the following relative solidity property for biexact groups in the measure equivalence setting.
% For simplicity, we state it in the orbit equivalence case (see also Theorem~\ref{thm: rel solid in ME} for an equivalent measure theoretical statement).}

It is an open question whether the above relative solidity property is satisfied by general biexact groups, and hence, this leads to the main technical difficulty for proving Theorem \ref{theorem.upf.biexact}.

\begin{conjecture}\label{conj}
Let  $\Gamma$ be a biexact group. If $\Gamma\car N$ is a trace preserving action and $P,Q\subset M:=N\rtimes\Gamma$ are commuting von Neumann subalgebras, then $P\prec_{M} N$ or $Q$ is amenable relative to $N$ inside $M$. 
\end{conjecture}

This problem originates from the pioneering work of Ozawa \cite{Oz04,Oz06} which states that Conjecture \ref{conj} is true whenever $N$ is a tracial abelian von Neumann algebra. 
Note that \cite{PoVa14II} shows that Conjecture \ref{conj} is true under the additional assumption that $\Gamma$ is weakly amenable.
Also, \cite[Proposition 7.3]{Iso19} shows that Conjecture \ref{conj} has a positive answer when $N$ satisfies the W$^*$CMAP. 
%(\textcolor{blue}{or CMAP} property?; other references? Chifan-Sinclair? Some of your work with Jesse?)

Despite the fact that Conjecture \ref{conj} is still open in its full generality, we are able to overcome this difficulty and prove Theorem \ref{theorem.upf.biexact} by showing a measure equivalence variant of Conjecture \ref{conj} holds true. 
For simplicity we state here an orbit equivalence version (see also Theorem~\ref{thm: rel solid in ME} for an equivalent measure theoretical statement).

%We prove a relative solidity property for biexact groups in the measure equivalence setting, but for simplicity we state it in the orbit equivalence case (see also Theorem~\ref{thm: rel solid in ME} for an equivalent measure theoretical statement). %See also Theorems \ref{cor: rel solid in ME} and \ref{lem: flip-rel solid}  for other technical variations of Theorem \ref{thm: rel solid} that are needed for proving Theorem \ref{theorem.upf.biexact}. 

%\textcolor{blue}{I think the notation from the next theorem should be changed.}

\begin{main}\label{thm: rel solid}
Let $\Gamma$ be a nonamenable biexact group and $\Lambda$, $\Sigma$ infinite groups.
Suppose $\Lambda\actson (Y,\nu)$ and $\Gamma\times \Sigma\actson (X, \mu)$ are orbit equivalent free ergodic p.m.p.\ actions
	and denote $M=L^\infty(X,\mu)\rtimes(\Gamma\times \Sigma)=L^\infty(Y,\nu)\rtimes\Lambda$.
		
Then for any subgroup $\Delta<\Lambda$, we have either $L^\infty(Y)\rtimes\Delta\prec_M L^\infty(X,\mu)\rtimes\Sigma$ or
	$L^\infty(Y)\rtimes C_{\Lambda}(\Delta)$ is amenable relative to $L^\infty(X,\mu)\rtimes\Sigma$ in $M$,
	where $C_\Lambda(\Delta)$ denotes the centralizer of $\Delta$ in $\Lambda$.	
%Equivalently, we have either $\cR({\Delta\actson X})$ intertwines in $\cR({\Lambda \actson Y})$ in $\cR$
%	or $\cR(C_{\Sigma}(\Delta)\actson X)$ is amenable relative to  $\cR({\Lambda \actson Y})$ in $\cR$.
\end{main}

%Note that Theorem \ref{thm: rel solid} provides the the equivalence relation resolution of Conjecture \ref{conj}. 

A remarkable progress has been obtained in von Neumann algebras and orbit equivalence by the emergence of Popa’s deformation/rigidity theory \cite{Po07A}. %A plethora of impressive structural results have been obtained for von Neumann algebras arising from large classes of countable groups $\Gamma$ and their measure preserving actions (see [Vae10, Ioa12, Ioa17]). 
For large classes of II$_1$ factors, many remarkable structural properties have been established, such as primeness, (infinite) unique prime factorization, (infinite) product rigidity, classification of normalizers of algebras, W$^*$-superrigidity, etc. (see the surveys \cite{Va10,Io12,Io17}).  Numerous of these findings have been achieved for von Neumann algebras that arise from groups that are biexact and weakly amenable by relying on Popa and Vaes' work \cite{PoVa14II} on the relative (strong) solidity property, see for example \cite{BeVa14,CIK13,CdSS15,KV15,DHI16,ChIo18,CU20,Iso19,Iso20,Dr19,CDK19,CDHK20,CD-AD23A,CDAD23B,DP22, CIOS23,CFQT24, DV24}. Many of these works do not depend on the relative strong solidity property of the groups, but only on the relative solidity property. Hence, in light of Conjecture \ref{conj}, a natural question  is to understand to what extent is weak amenability essential in these results. 
In this paper, we make progress on this question by applying Theorem \ref{thm: rel solid}  to generalize several rigidity results in the literature as outlined below. 

Our next main result, Theorem \ref{theorem.product.rigidity}, concerns orbit equivalence rigidity results for product actions. To put everything into context, note that \cite{CdSS15} shows that the group von Neumann algebra of a product of nonamenable i.c.c.\ weakly amenable and biexact groups remembers the product structure. The main result of \cite{Dr19} states that a similar phenomenon holds in the equivalence relation setting by showing that the orbit equivalence relation of a product of group actions of infinite weakly amenable, biexact, property (T) groups remembers the product structure. Since the group von Neumann algebra product ridigity from \cite{CdSS15} has been extended in \cite{CD-AD23A} to the class of nonamenable i.c.c.\ biexact groups, it remained open to show that the orbit equivalence product rigidity from \cite{Dr19} holds for infinite biexact, property (T) groups which are not necessarily weakly amenable. In the next theorem we affirmatively solve this problem.

%In our next main result, Theorem \ref{theorem.product.rigidity}, we show that the orbit equivalence relation of a product of group actions of biexact property (T) groups remembers the product structure. This result has been obtained in \cite{Dr19} under the additional assumption that the groups are also weakly amenable.

\begin{main} \label{theorem.product.rigidity}
Let $\Gamma_1,\dots, \Gamma_n$ be infinite biexact, property (T) groups and let $\Gamma=\times_{i=1}^n\Gamma$.
For each $i=1,\dots, n$, let $\Gamma_i\actson (X_i,\mu_i)$ be a free ergodic p.m.p.\ action and $\Gamma\actson (X,\mu)$ denotes the product action $\Gamma_1\times \cdots\times \Gamma_n\actson (X_1\times \cdots\times X_n, \mu_1\times\cdots \times \mu_n)$.

Let $\Lambda\actson (Y,\nu)$ be a free ergodic p.m.p.\ action of an i.c.c. groups that is stably orbit equivalent to $\Gamma\actson (X,\mu)$.
Then $\Lambda\actson Y$ is induced from a product action 
	$\Lambda_0:=\Lambda_1\times \cdots \times \Lambda_n\actson (Y_1\times \cdots Y_n, \nu_1\times \cdots \nu_n)$ of a subgroup $\Lambda_0<\Lambda$, such that
$\Gamma_i\actson X_i$ is stably orbit equivalent to $\Lambda_i\actson Y_i$ for each $i=1,\dots, n$.
\end{main}

 %In contrast to Theorem \ref{theorem.product.rigidity} and \cite{Dr19}, note that  
By using deep results from geometric group theory \cite{AMO07,Os06}  we point out the following class of biexact, property (T) groups that are not weakly amenable.    
\begin{examp}\label{examples}
Consider the wreath product $H= \mathbb Z \wr \mathbb F_n$ and notice that $H$ is a finitely generated biexact group \cite[Corollary 15.3.9.]{BrOz08}. By \cite[Theorem 1.1]{AMO07} there is a property (T) group $G$ such that $H<G$ with the property that $G$ is hyperbolic relative to $\{H\}$. Since $H$ is biexact, \cite[Theorem 1.1]{Oya23a} implies that $G$ is biexact as well. Note however that $G$ is not weakly amenable, since $H$ is not weakly amenable \cite{OzPo10I,Oz10}.

\end{examp}

We continue by improving several other results from the literature where one may cover nonamenable biexact groups that are not necessarily weakly amenable,
	including $\Z^2\rtimes\SL_2(\Z)$ \cite{Oza09}. See also  \cite{Oya23a, Oya23b} for more recent examples. The proofs of these results uses Theorem \ref{thm: rel solid} and rely on developing a classification result for commuting subalgebras in von Neumann algebras of infinite direct sums of biexact groups, see Corollary \ref{corollary.relative.solidity}.

\begin{itemize}
	%\item The weak amenability assumption in \cite[Theorem A]{Dr19} is removed in Theorem~\ref{theorem.product.rigidity},		which is an orbit equivalence rigidity result for product actions of biexact property (T) groups 	(that roughly says the product structure of a product action of such groups is preserved under orbit equivalence).
	\item The weak amenability assumption in \cite[Theorem A]{CU20} is removed in Theorem~\ref{main.infinite.product.rigidity},
		which concerns infinite product rigidity for i.c.c. biexact, property (T) groups
	(that roughly says that the direct sum structure is retained by the group von Neumann algebra of an infinite direct sum of such groups).
	\item Corollary~\ref{cor: infinite UPF} shows unique prime factorization for infinite tensor products of biexact (in the sense of \cite{DP22}) factors, eliminating the W*CBAP condition in \cite{Iso19}.
	\item Corollary~\ref{cor: fundamental group} relates the fundamental group of a product of orbit equivalence relations arising from actions of biexact groups to the product of their individual fundamental groups, 
		analogous to \cite[Theorem A]{Iso20} on the fundamental groups of ${\rm II}_1$ factors arising from weakly amenable and biexact groups.
\end{itemize}

Along with some of the above results on infinite direct sums/tensor products, another consequence of our techniques is the following measure equivalence rigidity for infinite direct sums of biexact groups, which may be seen as an infinite version of \cite[Theorem 4]{Sa09}.

\begin{main}\label{thm: infinite ME}
Let $\{\Gamma_n\}_{n\in \N}$ and $\{\Lambda_n\}_{n\in \N}$ be families of nonamenable discrete groups
	such that $\oplus_{n\in \N}\Gamma_n$ is measure equivalent to $\oplus_{n\in \N} \Lambda_n$.

If $\Gamma_n$ and $\Lambda_n$ are biexact for all $n\ge 1$,
	then $\Gamma_n$ is measure equivalent to $\Lambda_n$, up to a permutation of indices. 
\end{main}

In particular, this shows that $\oplus_{n\in \N}\F_2$ is not measure equivalent to $\oplus_{n\in \N}\F_{\infty}$ by \cite{Ga00}.

Another corollary of Theorem~\ref{thm: infinite ME} is to the classification of wreath product groups up to measure equivalence. By using deformation/rigidity theory, Chifan, Popa and Sizemore provided in \cite[Theorem 0.1]{CPS11} rigidity results for wreath product groups extending results from \cite[Theorem 7]{Sa09}. 
Under general assumptions on the groups  $A,B,\Gamma$ and $\Lambda$ (e.g. all the groups are i.c.c. property (T) groups), they proved that if $A\wr \Gamma$ is measure equivalent to $B\wr \Lambda$, then $A^{(\Gamma)}$ is measure equivalent to $B^{(\Lambda)}$ and $\Gamma$ is measure equivalent to $\Lambda$. Our next result complements the previous wreath product rigidity results and shows in addition that the groups $A$ and $B$ are measure equivalent.

\begin{mcor}\label{cor.ME.wreathproduct}
Let $A$ and $B$ be nonamenable biexact groups and $\Gamma$, $\Lambda$ nonamenable hyperbolic groups.

If $A\wr\Gamma$ is measure equivalent to $B\wr\Lambda$, then $\Gamma$ is measure equivalent to $\Lambda$ 
	and $A$ is measure equivalent to $B$.
\end{mcor}

To put our result into a better perspective, note that
Corollary \ref{cor.ME.wreathproduct} can  be contrasted with Tucker-Drob and Wr\'obel's measure equivalence flexibility reults for wreath product groups \cite{TDW24}. They showed in particular that if $A$ and $B$ are measure equivalent groups, then $A\wr\mathbb F_n$  is measure equivalent to $B\wr\mathbb F_n$ for any $n\ge 2$. Finally, Ioana and Tucker-Drob proved in particular that Corollary \ref{cor.ME.wreathproduct} holds whenever the groups $A,B,\Gamma,\Lambda$ are i.c.c. hyperbolic  \cite{ITD25}.

{\bf Organization of the paper.}
Besides the introduction, the paper has five other sections. In Section \ref{section.preliminaries} we have preliminaries, which shows in particular that 
 intertwining via measure equivalent couplings in the sense of Sako \cite{Sa09} is equivalent to
 intertwining  in the sense of Popa \cite{Po06A}. This is a key observation for proving measure equivalence of groups. In Section
\ref{sec: biexact} we recall the notion of biexact groups and von Neumann algebras \cite{BrOz08,DP22} and prove relative biexactness for infinite direct sum groups. In Section \ref{sec: rel solid} we prove relative solidity results such as Theorem \ref{thm: rel solid} and all its technical variations that are needed for Theorem \ref{theorem.upf.biexact}. The arguments used here are based on the framework developed in \cite{DP22} and make crucial use of certain approximation property coming from the measure equivalence assumption (see Remark~\ref{rem: technical}). Finally, in Sections \ref{sec: proof upf} and \ref{sec: infinite direct sum} we prove the remaining main results. We mention that the proof of Theorem \ref{theorem.upf.biexact} is inspired by the approach of \cite{Dri23} by relying on  the flip automorphism method from \cite{IM22}.

\textbf{Acknowledgements.} 
CD would like to express his gratitude to Stuart White for hosting his visit to Oxford University, during which the current work was started.
We are grateful to Stefaan Vaes for inspiring conversations and to Ionut Chifan for providing us the class of biexact groups from Example \ref{examples}.  We are thankful to Adrian Ioana and Stefaan Vaes for helpful comments that improved the exposition of the paper.
DD was partially supported by Engineering and Physical Sciences Research Council grant EP/X026647/1. 

%CD would like to thank Stefaan Vaes, from whom he learnt the question on whether $L\Gamma\ovt N$ is solid relative to $N$\footnote{i.e., for any von Neumann subalgebra $B\subset L\Gamma\ovt N$, one has $B\prec_{L\Gamma\ovt N} N$ or $B'\cap (L\Gamma\ovt N)$ is amenable relative to $N$ in $L\Gamma\ovt N$.} , for any nonamenable biexact group $\Gamma$ and ${\rm II}_1$ factor $N$, which partly motivated this project. 
%He is also grateful to Stuart White for...

\section{Preliminaries}\label{section.preliminaries}

\subsection{Terminology}

In this paper we consider {\it tracial von Neumann algebras} $(M,\tau)$, i.e. von Neumann algebras $M$ equipped with a faithful normal tracial state $\tau: M\to\mathbb C.$ This induces a norm on $M$ by the formula $\|x\|_2=\tau(x^*x)^{1/2},$ for all $x\in M$. We will always assume that $M$ is a {\it separable} von Neumann algebra, i.e. the $\|\cdot\|_2$-completion of $M$ denoted by $L^2(M)$ is separable as a Hilbert space.
We denote by $\mathcal Z(M)$ the {\it center} of $M$ and by $\mathcal U(M)$ its {\it unitary group}. %For two von Neumann subalgebras $P_1,P_2\subset M$, we denote by $P_1\vee P_2=W^*(P_1\cup P_2)$ the von Neumann algebra generated by $P_1$ and $P_2$. 

All inclusions $P\subset M$ of von Neumann algebras are assumed unital. We denote by $E_{P}:M\to P$ the unique $\tau$-preserving {\it conditional expectation} from $M$ onto $P$, by $e_P:L^2(M)\to L^2(P)$ the orthogonal projection onto $L^2(P)$ and by $\langle M,e_P\rangle$ the Jones' basic construction of $P\subset M$. We also denote by $P'\cap M=\{x\in M|xy=yx, \text{ for all } y\in P\}$ the {\it relative commutant} of $P$ in $M$ and by $\mathcal N_{M}(P)=\{u\in\mathcal U(M)|uPu^*=P\}$ the {\it normalizer} of $P$ in $M$.  We say that $P$ is {\it regular} in $M$ if the von Neumann algebra generated by $\mathcal N_M(P)$ equals $M$. 

The {\it amplification} of a II$_1$ factor $(M,\tau)$ by a number $t>0$ is defined to be $M^t=p(\mathbb B(\ell^2(\mathbb Z))\bar\otimes M)p$, for a projection $p\in \mathbb B(\ell^2(\mathbb Z))\bar\otimes M$ satisfying $($Tr$\otimes\tau)(p)=t$. Here Tr denotes the usual trace on $\mathbb B(\ell^2(\mathbb Z))$. Since $M$ is a II$_1$ factor, $M^t$ is well defined. Note that if $M=P_1\bar\otimes P_2$, for some II$_1$ factors $P_1$ and $P_2$, then there exists a natural identification $M=P_1^t\bar\otimes P_2^{1/t}$, for every $t>0.$

%Let $\Gamma\overset{\sigma}{\car} A$ be a trace preserving action of a countable group $\Gamma$ on a tracial von Neumann algebra $(A,\tau)$. For a subgroup $\Sigma<\Gamma$, we denote by $A^\Sigma=\{a\in A|\sigma_g(a)=a, \text{ for all } g\in\Sigma\}$, the subalgebra of elements of $A$ fixed by $\Sigma.$

Finally, let $I$ be a set. For any subset $J\subset I$ we denote its complement by $\widehat J=I\setminus J$. If $S=\{i\},$ we will simply write $\hat i$ instead of $\widehat {\{i\}}$. Also, given
any direct sum group $G=\oplus_{i\in I} G_i$ and any tensor product $\bar\otimes M_{i\in I}$, we will denote their subproduct supported on $J$ by $G_J=\oplus_{j\in J} G_j$ and $M_J=\bar\otimes_{j\in J} M_j$, respectively.

\subsection {Intertwining-by-bimodules} We next recall from  \cite [Theorem 2.1 and Corollary 2.3]{Po06B} the powerful {\it intertwining-by-bimodules} technique of Sorin Popa.

\begin {thm}[\!\!\cite{Po06B}]\label{corner} Let $(M,\tau)$ be a tracial von Neumann algebra and let $P\subset pMp, Q\subset qMq$ be von Neumann subalgebras. Let $\mathcal U\subset\mathcal U(P)$ be a subgroup such that $\mathcal U''=P$.

Then the following are equivalent:

\begin{enumerate}

\item There exist projections $p_0\in P, q_0\in Q$, a $*$-homomorphism $\theta:p_0Pp_0\rightarrow q_0Qq_0$  and a non-zero partial isometry $v\in q_0Mp_0$ satisfying $\theta(x)v=vx$, for all $x\in p_0Pp_0$.

\item There is no sequence $(u_n)_n\subset\mathcal U$ satisfying $\|E_Q(xu_ny)\|_2\rightarrow 0$, for all $x,y\in M$.

\end{enumerate}
\end{thm}

If one of the equivalent conditions of Theorem \ref{corner} holds true, we write $P\prec_{M}Q$, and say that {\it a corner of $P$ embeds into $Q$ inside $M$.}
Throughout the paper we will use the following notation.
\begin{itemize}
\item If $Pp'\prec_{M}Q$ for any non-zero projection $p'\in P'\cap pMp$, then we write $P\prec^{s}_{M}Q$.
\item If $P\prec_{M}Qq'$ for any non-zero projection $q'\in Q'\cap qMq$, then we write $P\prec^{s'}_{M}Q$.
\end{itemize}

We first recall the following intertwining result which is essentially contained in \cite[Section 3]{CD-AD23A}.

\begin{lem}\label{lemma.trick}
Let $\Lambda$ be a nontrivial countable i.c.c. group and denote $M=L(\Lambda)$. Let $\Delta:M\to M\ovt M$ be the $*$-homomorphism given by $\Delta(v_\lambda)=v_\lambda\otimes v_\lambda$, for any $\lambda\in \Lambda$.

Let $P\subset pMp$ be a regular von Neumann subalgebra and $\Sigma<\Lambda$ an infinite subgroup satisfying $P\prec_M L(\Sigma).$ If $Q\subset qMq$ is a von Neumann subalgebra such that $L(\Sigma)\prec_M Q$, then $\Delta(P)\prec_{M\ovt M}M\ovt Q$.

\end{lem}

\begin{proof}
We start the proof by considering   a Bernoulli action $\Lambda\curvearrowright B$ with tracial abelian base and denote $\mathcal M=B\rtimes\Lambda$.
We extend $\Delta$ naturally to a $*$-homomorphism $\Delta: \mathcal M\to \mathcal M\ovt \mathcal M$ by letting $\Delta(bv_\lambda)=bv_\lambda\otimes v_\lambda$ for all $b\in B$ and $\lambda\in\Lambda$.
The assumption implies that $P\prec_{\mathcal M}^{s'} B\rtimes\Sigma$. By using \cite[Lemma 2.3]{Dr19} it follows that  $\Delta(P)\prec_{\mathcal M\ovt\mathcal M}^{s'} \Delta(L(\Sigma))$. 
We further deduce that $\Delta(P)\prec_{\mathcal M\ovt\mathcal M}^{s'}\mathcal M\ovt L(\Sigma)$. By using  $L(\Sigma)\prec_M Q$,  \cite[Lemma 2.4]{Dr19} implies that $\Delta(P)\prec_{\mathcal M\ovt\mathcal M}\mathcal M\ovt Q$. By \cite[Lemma 3.3]{CD-AD23A} it follows that $\Delta(P)\prec_{\mathcal M\ovt\mathcal M} Q\ovt Q$. By \cite[Lemma 3.4]{CD-AD23A} it follows that $\Delta(P)\prec_{ M\ovt M} M\ovt Q$.
    \end{proof}

\subsection{Relative amenability}
A tracial von Neumann algebra $(M,\tau)$ is {\it amenable} if there is a positive linear functional $\Phi:\mathbb B(L^2(M))\to\mathbb C$ such that $\Phi_{|M}=\tau$ and $\Phi$ is $M$-{\it central}, meaning $\Phi(xT)=\Phi(Tx),$ for all $x\in M$ and $T\in \mathbb B(L^2(M))$. By Connes' celebrated work \cite{Co76}, it follows that $M$ is amenable if and only if $M$ is approximately finite dimensional.
 
We continue by recalling the notion of relative amenability introduced by 
Ozawa and Popa in \cite{OzPo10I}. Let $(M,\tau)$ be a tracial von Neumann algebra. Let $p\in M$ be a projection and $P\subset pMp,Q\subset M$ be von Neumann subalgebras. Following \cite[Definition 2.2]{OzPo10I}, we say that $P$ is {\it amenable relative to $Q$ inside $M$} if there exists a positive linear functional $\Phi:p\langle M,e_Q\rangle p\to\mathbb C$ such that $\Phi_{|pMp}=\tau$ and $\Phi$ is $P$-central. 

The proof of the following lemma is almost identical to the proof of \cite[Proposition 2.12]{CDD25} and hence we omit it.
\begin{lem}\label{lem: commultiplication and rel amen}
Let $\Gamma$ be a discrete group, $\Lambda<\Gamma$ a subgroup and $M$ a tracial von Neumann algebra with $\Gamma\actson M$ a trace preserving action.
Denote by $\Delta: M\to M\ovt L(\Gamma)$ the comultiplication map.
Then a von Neumann subalgebra $B\subset M\rtimes\Gamma$ is amenable relative to $M\rtimes\Lambda$ if and only if $\Delta(B)$ is amenable relative to $M\ovt L(\Lambda)$.
\end{lem}

\begin{lem}\label{lem: cond exp}
Given a discrete group $\Gamma$ and a tracial von Neumann algebra $(M,\tau)$, 
	let $\Gamma\actson^\sigma (M,\tau)$ be a trace preserving action.	
For a subgroup $\Lambda<\Gamma$,
	there exists a (semifinite) trace preserving conditional expectation $E: \langle M\rtimes\Gamma, e_{M\rtimes\Lambda}\rangle \to M\ovt \ell^\infty(\Gamma/\Lambda)$
	such that 
	$E_{\mid M\rtimes\Gamma}$ coincides with the conditional expectation from $M\rtimes\Gamma$ to $M$.

Moreover, for any $u\in \cU(M\rtimes\Gamma)$ for the form $u=\sum_{n\in \N}u_{t_n}p_n$, where $\{p_n\}_{n\in \N}$ is a partition of unity in $\cZ(M)$
	and $\{t_n\}_{n\in \N}\subset \Gamma$,
	we have $E$ is $\Ad(u)$-equivariant.
\end{lem}
\begin{proof}
Consider the u.c.p.\ map $\phi: \B(L^2M\ot \ell^2\Gamma)\to \B(L^2M)\ovt \ell^\infty\Gamma$ given by 
	$\langle(\phi(T))(t)\xi, \eta\rangle=\langle T\xi\otimes\delta_t, \eta\otimes\delta_t\rangle$
	for any $T\in \B(L^2M\ot\ell^2\Gamma)$, $t\in \Gamma$ and $\xi, \eta\in L^2M$.

Note that if $T\in JMJ'\cap \B(L^2M\ot\ell^2\Gamma)$,
	then $\phi(T)(t)\in M$ for any $t\in \Gamma$.
A similar computation shows that for any $s\in \Gamma$ and any $T\in \B(L^2M\ot\ell^2\Gamma)$ commutes with $Ju_sJ$.
	one has $\phi(T)(ts)=\phi(T)(t)$ for any $t\in \Gamma$.
It follows that
$$
E:=\phi_{\mid \langle M\rtimes\Gamma, e_{M\rtimes\Lambda}\rangle}:\langle M\rtimes\Gamma, e_{M\rtimes\Lambda}\rangle\to M\ovt \ell^\infty(\Gamma/\Lambda).
$$

Realizing $\ell^\infty(\Gamma/\Lambda)\subset \B(L^2M\ot \ell^2\Gamma)$ via $\delta_{t\Lambda}\mapsto u_te_{M\rtimes\Lambda} u_t^*$ for $t\in \Gamma$, it is routine to check that $E_{\mid M\ovt \ell^\infty(\Gamma/\Lambda)}=\id$.
One also checks that $E(\sum_{t\in F}a_tu_t)=a_e\ot 1_{\Gamma/\Lambda}$ for any $F\subset \Gamma$ finite, where $a_t\in M$,
	and the fact $E_{\mid M\rtimes\Gamma}=E^{M\rtimes\Gamma}_M$ follows from the normality of $E(\cdot)(t)$ for any $t\in \Gamma$.
		
To see $E$ is (semifinte) trace preserving, where $M\ovt \ell^\infty(\Gamma/\Lambda)$ is equipped with $\sum_{t\in \Gamma/\Lambda}\langle \cdot \hat 1\ot \delta_t, \hat 1\ot \delta_t\rangle$,
 	first compute that $\Tr(e_F T e_F)=\sum_{t\in F}\langle E(T) \hat 1\otimes \delta_{t}, \hat 1\ot \delta_{t}\rangle$ for any $T=au_ge_{M\rtimes\Lambda} bu_h$, where $a,b\in M$, $g,h\in \Gamma$, $F\subset \Gamma/\Lambda$ is a finite set and $e_F=\sum_{t\in F} u_te_{M\rtimes\Lambda} u_t^*$.
It follows that $\Tr(e_F \cdot e_F)=\sum_{t\in F} \langle \cdot \hat 1\ot \delta_t, \hat 1\ot \delta_t\rangle$ on $\langle M\rtimes\Gamma, e_{M\rtimes\Lambda}\rangle$
	and hence $\Tr(T)=\sum_{t\in\Gamma/\Lambda}   \langle E(T) \hat 1\ot \delta_t, \hat 1\ot \delta_t\rangle$
	for any $T\in \langle M\rtimes\Gamma, e_{M\rtimes\Lambda}\rangle$.

Lastly, for any $u\in \cU(M\rtimes\Gamma)$ of the form $u=\sum_{n\in \N}u_{t_n}p_n$, $T\in \langle M\rtimes\Gamma, e_{M\rtimes\Lambda}\rangle$, 
	$\xi,\eta\in \hat M$ and $g\in \Gamma$, we compute
\[
\begin{aligned}
	\langle \phi(u^* T u) \xi\ot \delta_g, \eta \ot \delta_g\rangle
	&=\sum_{m,n\in \N} \langle T  (\sigma_{t_n}(p_n\xi) \otimes \delta_{t_ng}), \sigma_{t_m}(p_m\eta)\ot\delta_{t_mg}\rangle\\
	&=\sum_{m,n\in \N} \langle T (J\sigma_{g^{-1}}(p_n)J) (\sigma_{t_n}(\xi)\otimes\delta_{t_ng}), (J\sigma_{g^{-1}}(p_m) J)(\sigma_{t_m}(\eta)\otimes\delta_{t_m g})\rangle\\
	&=\sum_{n\in \N} \langle T  (\sigma_{t_n}(p_n\xi) \otimes \delta_{t_ng}), \sigma_{t_n}(p_n\eta)\ot\delta_{t_ng}\rangle\\
	&=\langle \phi(T) u(\xi\otimes\delta_g), u(\eta\ot \delta_g)\rangle.
\end{aligned}
\]
\end{proof}

\begin{lem}\label{lem: central state}
Let $\Gamma$ be a discrete group with a subgroup $\Lambda$, $(M,\tau)$ a tracial von Neumann algebra and $\Gamma\actson (M,\tau)$
	a trace preserving action and $p\in \cZ(M)$ a nonzero projection.
Suppose $\cG<\cN_{p(M\rtimes\Gamma )p}(pMp)$ is a subgroup containing elements of the form $\sum_{n\in \N} u_{t_n} p_n$,
	where $\{p_n\}_{n\in \N}$ is a partition of $p$ in $\cZ(M)$ and $\{t_n\}_{n\in \N}\subset \Gamma$.

Then $N:=\{pMp,\cG\}''$ is amenable relative to $M\rtimes\Lambda$ in $M\rtimes\Gamma$ if and only if 
	there exists a $\cG$-invariant state $\varphi$ on $M\ovt \ell^\infty(\Gamma/\Lambda)$ with $\varphi_{\mid pMp\ot 1}=\tau_{pMp}$.
\end{lem}

\begin{proof}
The only if direction is clear. To see the if direction, 
	consider $\omega:=\varphi\circ E: \langle M\rtimes\Gamma, e_{M\rtimes\Lambda}\rangle\to \C$,
	where $E$ is from the previous lemma.
Note that $\omega$ is $\cG$-invariant as $E$ is $\cG$-equivariant.
Similarly, we have $\omega$ is $pMp$-central.
Since $\omega_{\mid p(M\rtimes\Gamma)p}=\tau_{p(M\rtimes\Gamma)p}$,
	a standard argument shows that $\omega$ is indeed $N$-central.
\end{proof}

\subsection{Intertwining and relative amenability via measure equivalence}\label{sec: ME}
The main goal of this subsection is to discuss an measure equivalence counterpart of the von Neumann algebra intertwining notion \cite{Po06B} and relative amenability one \cite{OzPo10I}. Throughout this subsection, we assume the following notation.

	Let $\Gamma$ and $\Lambda$ be measure equivalent groups.  Then there exists a standard measure space $(\Omega,\mu)$ (called an ME-coupling between $\Gamma$ and $\Lambda$)
	with commuting measure preserving actions of $\Gamma$ and $\Lambda$
	such that each action admits a finite measure fundamental domain.
 Let $\Gamma_0<\Gamma$, $\Lambda_0<\Lambda$ be subgroups.
Denote by $X=\Omega^\Lambda$ and we may identify $(\Omega,\mu)$ with $(X\times \Lambda, \mu_{\mid X}\times c)$,
	where $c$ is the counting measure on $\Lambda$ and $X$ is identified with a $\Lambda$-fundamental domain in $\Omega$.
Note that we may identify $\Omega^{\Lambda_0}$ with $X\times \Lambda/\Lambda_0$, equipped with the measure $\mu_{\mid X}\times c_{\Lambda/\Lambda_0}$,
	which we denote by $\mu_{\Lambda_0}$.

We continue by recalling from  \cite[Definition 2]{Sa09} the following intertwining relation in the measure equivalence framework.

\begin{defn}[{\cite{Sa09}}]
We say that   $\Gamma_0$ intertwines into $\Lambda_0$ via the coupling $\Omega$, denoted by $\Gamma_0\preccurlyeq_\Omega \Lambda_0$
	if there exists a non-null measurable subset $E\subset \Omega$ that is $\Gamma_0\times \Lambda_0$-invariant 
	with $\mu_{\Lambda_0}(E)<\infty$.
\end{defn}
    
Inspired by \cite[Theorem 3.5, (ii)]{Hay24}, we are considering the following relative amenability notion in the framework of measure equivalence.

\begin{defn}
    We say that $\Gamma_0$ is amenable relative to $\Lambda_0$ via the coupling $\Omega$
	if there exists a $\Gamma_0$-invariant state $\nu$ on $L^\infty(\Omega^{\Lambda_0}, \mu_{\Lambda_0})$
	such that $\nu_{X}=\mu_{\mid X}$,
	where $\nu_X(1_F)=\nu(1_{F\times \Lambda/\Lambda_0})$ for any measurable subset $F\subset X$.
\end{defn}

%Recall from \cite[Definition 2]{Sa09} 	that one says $\Gamma_0$ intertwines into $\Lambda_0$ via the coupling $\Omega$, denoted by $\Gamma_0\preccurlyeq_\Omega \Lambda_0$
%	if there exists a non-null measurable subset $E\subset \Omega$ that is $\Gamma_0\times \Lambda_0$-invariant 
%	with $\mu_{\Lambda_0}(E)<\infty$.
%We say $\Gamma_0$ is amenable relative to $\Lambda_0$
%	if there exists a $\Gamma_0$-invariant state $\nu$ on $L^\infty(\Omega^{\Lambda_0}, \mu_{\Lambda_0})$
%	such that $\nu_{X}=\mu_{\mid X}$,
%	where $\nu_X(1_F)=\nu(1_{F\times \Lambda/\Lambda_0})$ for any measurable subset $F\subset X$.
%We remark that this formulation of relative amenability agrees with the consideration in \cite[Theorem 3.5, (ii)]{Hay24}.

%These notions were inspired by their counterparts in von Neumann algebras, which we recall now. Let $(M,\tau)$ be a tracial von Neumann algebra, $P\subset pMp$ and $Q\subset M$ von Neumann subalgebras, where $p\in M$ a nonzero projection. By \cite{Po06B}, we say $P\prec_M Q$ if there exists a nonzero projection $f\in p \langle M, e_Q\rangle p$ that commutes with $P$ and $\Tr_Q(f)<\infty$,
%	where $\Tr_Q$ denotes the canonical (semifinite) trace on $\langle M, e_Q\rangle$.
%    Moreover, if $Pp'\prec_M Q$ for any nonzero projection $p'\in pMp$, then we write $P\prec^s_M Q$.
%Following \cite{OzPo10I}, we say $P$ is amenable relative to $Q$ in $M$ if there exists a $P$-central positive linear functional 
%	$\varphi: p\langle M, e_Q\rangle p\to \C$ such that $\varphi_{\mid pMp}=\tau$

Continuing in the above setting, by \cite{Fu99A}
	we may assume that $X\subset Y$ and $\cR({\Gamma\actson X})=\cR(\Lambda\actson Y)\cap (X\times X)$
	by replacing $\Lambda$ with $\Lambda\times \Z_d$ and $\Omega^\Gamma$ with $\Omega^\Gamma\times \Z_d$ for some $d\in \N$,
	where $Y=\Omega^\Gamma$.
It follows that we have $(L^\infty(X)\subset L^\infty(X)\rtimes\Gamma)=(pL^\infty(Y)\subset p(L^\infty(Y)\rtimes\Lambda)p)$ with
	$p=1_X\in L^\infty(Y)$.
In the following, we denote by $\{u_t\mid t\in \Gamma\}$ and $\{v_g\mid g\in \Lambda\}$ the canonical unitaries in 
	$L^\infty(X)\rtimes\Gamma=:A\rtimes \Gamma$ and $L^\infty(Y)\rtimes\Lambda=:B\rtimes\Lambda=:M$
	associated with $\Gamma$ and $\Lambda$, respectively.
	
\begin{prop}\label{lem: ME intertwining}
Using the above notations, the following statements hold.
\begin{enumerate}
	\item One has $\Gamma_0\preccurlyeq_\Omega \Lambda_0$ if and only if $A\rtimes\Gamma_0\prec_M B\rtimes\Lambda_0$
	if and only if $L\Gamma_0\prec_M B\rtimes\Lambda_0$. \label{item: ME intertwining}
	\item The subgroup $\Gamma_0$ is amenable relative to $\Lambda_0$ via $\Omega$ 
	if and only if $A\rtimes \Gamma_0$ is amenable relative to $B\rtimes\Lambda_0$ in $M$
	if and only if $L\Gamma_0$ is amenable relative to $B\rtimes\Lambda_0$ in $M$. \label{item: ME rel amen}
\end{enumerate}
\end{prop}
\begin{proof}
We may identify $L^\infty(\Omega^{\Lambda_0})=p L^\infty(Y)\ovt \ell^\infty(\Lambda/\Lambda_0)\subset p\langle M, e_{B\rtimes\Lambda_0}\rangle p$
	and it is not difficult to check that this identification is (semifinite) trace preserving. 
It follows that $\Gamma_0\preccurlyeq_{\Omega}\Lambda_0$ implies $A\rtimes\Gamma_0\prec_M B\rtimes\Lambda_0$,
	and hence $L\Gamma_0\prec_M B\rtimes\Lambda_0$.
If $f\in p\langle M, e_{B\rtimes\Lambda_0}\rangle p$ is a nonzero projection with finite trace commuting with $L\Gamma_0$,
	then apply Lemma~\ref{lem: cond exp} yields $E(f)\in L^\infty(X)\ovt \ell^\infty(\Lambda/\Lambda_0)=L^\infty(\Omega^{\Lambda_0})$
	a nonzero positive element with finite trace.
Moreover, $E(f)$ is $\Gamma_0$-invariant as $E$ is $[\cR({\Lambda\actson Y})]$-equivariant.
Taking an appropriate spectral projection of $E(f)$ yields $\Gamma_0\preccurlyeq_\Omega \Lambda_0$.

To see (\ref{item: ME rel amen}), note that the inclusion $L^\infty(\Omega^{\Lambda_0})\subset 
	p\langle M, e_{B\rtimes\Lambda_0}\rangle p$ shows that
	if $L\Gamma_0$ is amenable relative to $B\rtimes\Lambda_0$ then $\Gamma_0$ is amenable relative to $\Lambda_0$ via $\Omega$.
The rest follows from Lemma~\ref{lem: central state}.	
\end{proof}

Finally, we record the following elementary lemma.

\begin{lem}\label{lem: commuting intertwine}
Let $\Gamma$ and $\Lambda$ be countable groups that are measure equivalent via an ergodic coupling $(\Omega,\mu)$.
Then the following hold.
\begin{enumerate}
	\item Suppose $\Gamma_1, \Gamma_2<\Gamma$ and $\Sigma<\Lambda$ are subgroups such that $\Gamma_1$ and $\Sigma$ are normal.
	 If $\Gamma_1$ and $\Gamma_2$ are commuting and $\Gamma_i\preccurlyeq_{\Omega}\Sigma$ for $i=1,2$, 
	 then $\Gamma_1\Gamma_2\preccurlyeq_{\Omega}\Sigma$.
	\item Suppose $\Gamma_1, \Gamma_2<\Gamma$ and $\Sigma<\Lambda$ are subgroups  such that $\Gamma_1$ and $\Sigma$ are normal.
	If $\Gamma_1\preccurlyeq_{\Omega}\Lambda$ and $\Lambda\preccurlyeq_{\Omega}\Gamma_2$, then $\Gamma_1\subset F \Gamma_2$ for some finite subset $F\subset \Gamma$.
\end{enumerate}	
\end{lem}
\begin{proof}
We follow the proof of \cite[Lemma 33]{Sa09}.
Let $e_i\in L^\infty(\Omega)^{\Gamma_i\times \Sigma}$ be a nonzero projection.
By ergodicity of $\Omega$, we may assume $e_1 e_2\neq 0$ by replacing $e_1$ with $t\cdot e_1\cdot g$ for some $t\in \Gamma/\Gamma_1$ and $g\in \Lambda/\Sigma$.

To see (1), we may further assume $\mu_{\Sigma}(e_i)<\infty$ for $i=1,2$.
Set $\cC={\rm conv}\{t\cdot e_1\mid t\in \Gamma_2\}\subset L^2( L^\infty(\Omega)^{\Sigma}, \mu_{\Sigma})$
	and pick $f$ to be the minimal norm element in the closure of $\cC$.
Since $\mu_{\Sigma}((t\cdot e_1) e_2)=\mu_{\Sigma}(t\cdot (e_1 e_2))=\mu_{\Sigma}(e_1 e_2)>0$,
	one has $f\neq 0$.
Moreover, we may assume $f$ is a projection by taking an appropriate spectral projection. 
It follows that $f$ implements $\Gamma_1\Gamma_2\preccurlyeq_{\Omega}\Sigma$.

The argument is similar for (2). Indeed, we have $\mu_{\Sigma}(e_1)<\infty$ and $\mu_{\Gamma_2}(e_2)<\infty$ by assumption.
Let $p$ be the minimal norm element in the closure of ${\rm conv}\{t\cdot e_2\mid t\in \Gamma_1\}\subset L^2(L^\infty(\Omega)^{\Gamma_2},\mu_{\Gamma_2})$.
To see $p$ is nonzero, observe
$\mu_{\Sigma}( e_1 (t\cdot e_2))=\mu_{\Sigma}(t\cdot (e_1e_2))>0$.
It follows that there exists a nonzero $\Gamma_1$-invariant projection in $\ell^\infty(\Gamma/\Gamma_2)$ with finite trace,
	as desired.
\end{proof}

\subsection{From intertwining to measure equivalence}
All results in this subsection are due to Sako \cite{Sa09}.
These results, especially when combined with Popa's intertwining techniques through Proposition~\ref{lem: ME intertwining},
	give efficient ways to conclude measure equivalence from intertwining.
We present here a streamlined proof for ease of reference and completeness. 

\begin{note}\label{note: ME intertwine}
Let $\Gamma$ and $\Lambda$ be countable groups that are measure equivalent via an ergodic coupling $(\Omega,\mu)$,
	assume $\Gamma_1\lhd \Gamma$, $\Lambda_1\lhd \Lambda$ are normal subgroups
	and put $\Gamma_2=\Gamma/\Gamma_1$, $\Lambda_2=\Lambda/\Lambda_1$.

For any projection $p\in L^\infty(\Omega^{\Gamma_1},\mu_{\Gamma_1} )$, 
    we may consider $\iota_{\Gamma_1}(p)$, a $\mathbb N\cup\{\infty\}$-valued measurable function on $Y$
    given by $\iota_{\Gamma_1}(p)(y)=\sum_{t\in \Gamma_2}p(t,y)$,
    where $(Y,\lambda)=(\Omega^{\Gamma}, \mu_{\Gamma})$ and $(\Omega^{\Gamma_1}, \mu_{\Gamma_1})$ is identified with $(\Gamma_2\times Y, c\times \lambda)$.
Note that $\iota_{\Gamma_1}$ is $\Lambda$-equivariant as $\Gamma_1$ is normal
    and $\mu_{\Gamma_1}(p)=\lambda(\iota_{\Gamma_1}(p))$.
Similarly, we consider $\iota_{\Lambda_1}(q)$ as a $\mathbb N\cup\{\infty\}$-valued measurable function on $X$ for any projection $q\in L^\infty(\Omega^{\Lambda_1},\mu_{\Lambda_1})$.
\end{note}

\begin{lem}[{\cite[proof of Theorem 30]{Sa09}}, cf. {\cite[Proposition 3.1]{DHI16}}]\label{lem: flip intertwining ME}
With Notation~\ref{note: ME intertwine}, assume in addition that $\Gamma=\Gamma_1\times \Gamma_2$ and $\Lambda=\Lambda_1\times \Lambda_2$.
If $\Gamma_i\preccurlyeq_\Omega \Lambda_i$ for $i=1,2$, then $\Gamma_i\sim_{\rm ME} \Lambda_i$ for $i=1,2$
\end{lem}
\begin{proof}
We denote by $e\in L^\infty(\Omega^{\Gamma_1\times \Lambda_1})$ and $f\in L^\infty(\Omega^{\Gamma_2\times \Lambda_2})$
	be nonzero projections with $\mu_{\Lambda_1}(e)<\infty$ and $\mu_{\Lambda_2}(f)<\infty$.
Since $\rho(\iota_{\Lambda_1}(e))=\mu_{\Lambda_1}(e)<\infty$, we may assume $\iota_{\Lambda_1}(e)$ is bounded. 
Similarly, we may assume $\iota_{\Lambda_2}(f)$ is bounded.
Moreover by ergodicity of $\Gamma \times \Lambda\actson (\Omega, \mu)$, we may assume $ef\neq 0$.

Observe that 
$$
\int_\Omega efd\mu=\int_{X} \iota_{\Lambda_1}(e)\iota_{\Lambda_2}(f)d\rho
$$
and hence 
$
0<\int_\Omega efd\mu=\int_{X} \iota_{\Lambda_1}(e)\iota_{\Lambda_2}(f)d\rho<\infty.
$
Indeed, this is clear when $e=1_{X_s\times \Lambda_1\times s}$ and $f=1_{X_t\times t\times \Lambda_2}$ for $s\in\Lambda_2$, $t\in\Lambda_1$ and $X_s, X_t\subset X$ measurable subsets.
The general case follows by linearity and normality of the integration.
For the same reason, we have
$
0<\int_\Omega efd\mu=\int_{X} \iota_{\Lambda_1}(e)\iota_{\Lambda_2}(f)d\rho<\infty.
$

Set $g:=\iota_{\Gamma_1}(e)\iota_{\Gamma_2}(f):Y\to \mathbb N\cup\{\infty\}$ and denote by $Y_0$ the support of $g$.
Note that $\iota_{\Gamma_1}(e)(x)\leq \iota_{\Gamma_1}(e)(x) \iota_{\Gamma_2}(f)(x)=g(x)$ for any $x\in Y_0$,
	as $\iota_{\Gamma_2}(f)(x)\geq 1$.
Moreover, $\iota_{\Gamma_1}(e)1_{Y_0}$ is a nonzero function as $g$ is nonzero.
Thus $0<\int_{Y_0} \iota_{\Gamma_1}(e) d\lambda <\infty$
and there exists some positive integer $k$ with $\lambda((\iota_{\Gamma_1}(e)1_{Y_0})^{-1}(k))>0$.
Set $Y_k=(\iota_{\Gamma_1}(e))^{-1}(k)$, which is $\Lambda_1$-invariant,
	and then $\tilde e:=e1_{\Gamma_2\times Y_k}\in L^\infty(\Omega^{\Gamma_1})$ is $\Lambda_1$-invariant and
$$
\int_{\Omega^{\Gamma_1}} e 1_{\Gamma_2\times Y_k} d\mu_{\Gamma_1}=\int_{Y_k}\iota_{\Gamma_1}(e)d\lambda\leq k\lambda(Y),
$$
i.e., $\tilde e$ is finite in both $\mu_{\Gamma_1}$ and $\mu_{\Lambda_1}$ and hence $\Gamma_1\sim_{\rm ME} \Lambda_1$.
Similarly, we have $\Gamma_2\sim_{\rm ME}\Lambda_2$.
\end{proof}
\begin{lem}[{\cite[Proposition 28]{Sa09}, cf. \cite[Proposition 3.9]{Spa23}}]\label{lem: normal subgroup intertwine}
Assume Notation~\ref{note: ME intertwine}.
If $\Gamma_1\preccurlyeq_\Omega \Lambda_1$ and $\Lambda_1\preccurlyeq_\Omega \Gamma_1$, then $\Gamma_i\sim_{\rm ME} \Lambda_i$ for $i=1,2$.
\end{lem}
\begin{lem}
Assume Notation~\ref{note: ME intertwine}.
If there exists a nonzero projection $e\in L^\infty(\Omega^{\Gamma_1\times \Lambda_1})$ with $\mu_{\Lambda_1}(e)<\infty$, 
	then there exists a finite subgroup $F<\Lambda_2$ and a nonzero projection $f\in L^\infty(\Omega^{\Gamma_1\times \Lambda_1})$
	such that $\mu_{\Lambda_1}(f)<\infty$, $F$ acts trivially on $fL^\infty(\Omega^{\Gamma_1\times \Lambda_1})$, $(t\cdot f)f=0$ for any $t\in \Lambda_2\setminus F$ and
	$\vee_{t\in \Lambda_2} t\cdot f=1_{\Omega^{\Gamma_1\times \Lambda_1}}$
\end{lem}
\begin{proof}
Consider
$$
\{n\in \N_+ \mid \exists\ \Gamma_1 \text{-invariant\  projection}\  q\in L^1(\Omega^{\Lambda_1},\mu_{\Lambda_1})\ s.t.\ \rho(\iota_{\Lambda_1}(q)^{-1} (n))>0\},
$$
which is nonempty by the existence of $e$ and let $k$ be the minimal element in this set.
Thus we have a projection $p\in L^\infty(\Omega^{\Gamma_1\times \Lambda_1})$ such that $k\in {\rm Ran}(\iota_{\Lambda_1}(p))$.
We may replace $p$ with $p1_{E\times \Lambda}$, which is also $\Gamma_1\times \Lambda_1$-invariant,
	where $E\subset X$ the preimage of $k$ under $\iota_{\Lambda_1}(p)$.
	
Denote by $\Lambda_2^k$ the collection of subsets of $\Lambda_2$ of size $k$.
For any $K\in \Lambda_2^k$, we consider the measurable subset $X_K=\{x\in X\mid \sum_{t\in K}p(x, t)=k\}$,
	where we view $p\in L^\infty(\Omega^{\Lambda_1})=L^\infty(X\times \Lambda_2)$.
Since ${\rm Ran}(\iota_{\Lambda_1}(p))=\{0,k\}$, we have $p=\sum_{K\in \Lambda_2^k} 1_{X_K\times \Lambda_1\times K}$
	and $\rho(X_{K_1}\cap X_{K_2})=0$ for any $K_1, K_2\in \Lambda_2^k$ with $K_1\neq K_2$,
	and thus $\rho(\cup_{K\in \Lambda_2^k} X_K)=\mu_{\Lambda_1}(p)/k$.

Set $\cK=\{K\in \Lambda_2^k\mid \rho(X_K)>0\}$. For each $K\in \cK$ and $t\in \Lambda_2$, 
	we have $tK\cap K$ is either $K$ or $\emptyset$
	as $\iota_{\Lambda_1}((t\cdot p)p)1_{X_K}=|tK\cap K|1_{X_K}$
	and by the minimality of $k$ one has $|tK\cap K|=0$ or $k$.
It follows that $tK=K$ if $t\in  K K^{-1}$ and $tK\cap K=\emptyset$ otherwise.
We may assume there exists some $F\in \cK$ containing the identity element as we	
	may replace $p$ with $t^{-1}\cdot p$ for some $t\in F$,
	and it follows $F$ is a subgroup.

Now consider $f=\wedge_{t\in F} t\cdot p=\sum_{K\in \cK} (\wedge_{t\in F} 1_{X_K\times \Lambda_1\times tK})$.
For $K\in \cK$, one has $\wedge_{t\in F} 1_{X_K\times \Lambda_1\times tK}\neq 0$
	if and only if $tK=K$ for each $t\in F$,
	i.e., $K=Fs$ for any $s\in K$.
It follows that we may write $f=\sum_{s\in I} 1_{X_{Fs}\times \Lambda_1\times Fs}$
	for some subset $I\subset \Lambda_2$.
	
Utilizing the minimality of $k$, one checks that $F$ acts trivially on $fL^\infty(\Omega^{\Gamma_1\times \Lambda_1})$
	and $(t\cdot f)f=0$ for any $t\in \Lambda_2\setminus F$.
Since $f\leq p$, we also have $\mu_{\Lambda_1}(f)<\infty$.
	
We claim that there exists a projection $\tilde f\in L^\infty(\Omega^{\Gamma_1\times \Lambda_1})$ with the same aforementioned properties
	and $\vee_{t\in \Lambda_2}t\cdot \tilde f=1_{\Omega^{\Gamma_1\times \Lambda_1}}$.
Indeed, let $X_0:=\cup_{s\in I} X_{Fs}\subset X$, which is $\Gamma_1$-invariant as $f$ is $\Gamma_1$-invariant.
If $X_0\neq X$, we may find some $g\in \Gamma_2$ such that $U=g\cdot X_0\cap (X\setminus X_0)$ has positive measure.
Set $f_g=f1_{g^{-1}\cdot U\times \Lambda}=\sum_{s\in I} 1_{X_{Fs}\cap g^{-1}\cdot U\times \Lambda_1\times Fs}\in L^\infty(\Omega^{\Gamma_1\times \Lambda_1})$ and $\tilde f:=f+g\cdot f_g$. 
One then checks that $\iota_{\Lambda_1}(\tilde f)=k 1_{X_0\cup U}$, 
	$F$ acts trivially on $f_g L^\infty(\Omega^{\Gamma_1\times \Lambda_1})$
	and $(t\cdot \tilde f)\tilde f=0$ for any $t\in \Lambda_2\setminus F$ since
	$(t\cdot f) (g\cdot f_g)\leq 1_{X_0\times \Lambda}1_{U\times \Lambda}=0$.
By a maximality argument, we may assume $\iota_{\Lambda_1}(\tilde f)=k1_X$, as desired.
\end{proof}

\begin{proof}[Proof of Lemma~\ref{lem: normal subgroup intertwine}]
By assumption, there exists projections $f_0,e_0\in L^\infty(\Omega^{\Gamma_1\times \Lambda_1})$ such that 
	$\mu_{\Lambda_1}(f_0)<\infty$ and $\mu_{\Gamma_1}(e_0)<\infty$.
We may assume $r:=e_0 f_0\neq 0$ by ergodicity of $\Omega$ and hence $\mu_{\Gamma_1}(r) + \mu_{\Lambda_1}(r)<\infty$,
	which implies that $\Gamma_1\sim_{\rm ME} \Lambda_1$ via the coupling $rL^\infty(\Omega)$.

To see $\Gamma_2\sim_{\rm ME} \Lambda_2$, note that
	the above lemma yields from the existence of $f_0$ a nonzero projection $f\in L^\infty(\Omega^{\Gamma_1\times \Lambda_1})$ 
	and a finite subgroup $F<\Lambda_2$ such that $\mu_{\Lambda_1}(f)<\infty$, $F$ acts trivially on $fL^\infty(\Omega^{\Gamma_1\times \Lambda_1})$, 
		$(t\cdot f)f=0$ for any $t\in \Lambda_2\setminus F$
		and $\vee_{t\in \Lambda_2}t\cdot f=1$.
Similarly, we have a nonzero projection $e\in L^\infty(\Omega^{\Gamma_1\times \Lambda_1})$ and a finite subgroup $E<\Gamma_2$ with corresponding properties.

By ergodicity of $(\Omega,\mu)$, we may assume $ef\neq 0$.
Consider the Radon-Nikodym derivative $d\mu_{\Lambda_1}/d\mu_{\Gamma_1}$, which is $\Gamma\times \Lambda$-invariant as both $\mu_{\Gamma_1}$ and $\mu_{\Lambda_1}$ are,
	and hence must be a constant by ergodicity of $\Omega$.
Moreover as $0<\mu_{\Gamma_1}(ef),\mu_{\Lambda_1}(ef)<\infty$, one has $d\mu_{\Lambda_1}/d\mu_{\Gamma_1}$ equals to some positive finite constant
	and thus we may not need to distinguish $\mu_{\Gamma_1}$ and $\mu_{\Lambda_1}$ as measures on $L^\infty(\Omega^{\Gamma_1\times \Lambda_1})$.

Let $N_F<\Lambda_2$ be the normalizer of $F$ in $\Lambda_2$ and $q=\vee_{t\in N_F} t\cdot f$ and clearly we have an action of $N_F/F$ on $qL^\infty(\Omega^{\Gamma_1\times \Lambda_1})$ with $e$ being a fundamental domain.
We claim that $q$ is $\Gamma$-invariant.  
Indeed, we observe that $q\in L^\infty(\Omega^{\Gamma_1\times \Lambda_1})$ is the maximal projection such that $F$ acts trivially $qL^\infty(\Omega^{\Gamma_1\times \Lambda_1})$.
This is because for any $t\in \Lambda_2\setminus N_F$, we may find some $g\in F$ such that $t^{-1} g t\not\in F$ and hence $t^{-1}gt \cdot e$ is orthogonal to $e$, which entails that $g\cdot (t\cdot e)\perp (t\cdot e)$.
Then the claim follows directly from this characterization of $q$.																					
We may thus write $q=1_{\Gamma\times Y_q}$ for some $Y_q\subset Y$.

Next we show $N_F<\Lambda_2$ is finite index. 
Take a transversal $\{t_i\}_{i\in I}=\Lambda_2/ N_F$ and notice that 
	$t_i\cdot q\perp t_j\cdot q$.
It follows that $\{t_i\cdot Y_q\}_{i\in I}$ is a family of pairwise disjoint subsets of $Y$ with the same measure.
As $Y$ is a finite measure space, one has $|I|<\infty$.

In summary, we have an action $N_F/F\actson qL^\infty(\Omega^{\Gamma_1\times \Lambda_1})$ with a finite measure fundamental domain $f$, $N_F<\Lambda_2$ is of finite index and
	$q$ is a $\Gamma\times N_F$-invariant projection.
Replacing $e$ with $qe$, we may repeat the same argument to conclude that there exists a $N_K\times N_F$-invariant projection $p\in qL^\infty(\Omega^{\Gamma_1\times \Lambda_1})$ such that $N_k/K\actson pL^\infty(\Omega^{\Gamma_1\times \Lambda_1})$ has a finite measure fundamental domain $qe$ and $[\Gamma_2, N_K]<\infty$, where $N_k$  is the normalizer of $K$ in $\Gamma_2$.

Therefore, $pL^\infty(\Omega^{\Gamma_1\times \Lambda_1})$ is an ME-coupling between $N_F/F$ and $N_K/K$
	with a $N_F/F$-fundamental domain $pf$ and $N_k/K$-fundamental domain $qe$.
It follows that $\Gamma_2\sim_{\rm ME} \Lambda_2$.
\end{proof}

\section{Biexact groups and von Neumann algebras}\label{sec: biexact}

\subsection{The small-at-infinity boundary and boundary pieces}\label{sec: boundary pieces}
The notion of small-at-infinity compatification of a discrete group $\Gamma$ was introduced by Ozawa \cite{Oz04, BrOz08}.
More precisely, given a discrete group $\Gamma$, the small-at-infinity compatification $\overline\Gamma^s$ is described by
$$
C(\overline\Gamma^s)=\{f\in \ell^\infty\Gamma\mid f-R_tf\in c_0\Gamma,\ \forall t\in \Gamma\}\subset \ell^\infty\Gamma,
$$
where $R_tf(\cdot)=f(\cdot t)$ is the right translation by $\Gamma$ on $\ell^\infty\Gamma$.

The notion of the small-at-infinity boundary for von Neumann algebras developed in \cite{DKEP22, DP22}, 
	is a noncommutative analogue of the above notion and we recall it now.
Let $M$ be a tracial von Neumann algebra.
An $M$-boundary piece $\X$ is a hereditary ${\rm C}^*$-subalgebra $\X\subset\B(L^2M)$
	such that $M\cap M(\X)\subset M$ and $JMJ\cap M(\X)\subset JMJ$ are weakly dense,
	and $\X\neq \{0\}$,
	where $M(\X)$ denotes the multiplier algebra of $\X$.
For convenience, we will always assume $\X\neq \{0\}$.
Given an $M$-boundary piece $\X$, define $\K_\X^L(M)\subset \B(L^2M)$ to be the $\|\cdot\|_{\infty,2}$ closure of $\B(L^2M)\X$,
	where $\|T\|_{\infty,2}=\sup_{a\in (M)_1}\|T\hat a\|$
		and $(M)_1=\{a\in M\mid \|a\|\leq 1\}$.
Set $\K_\X(M)=\K_\X^L(M)^*\cap \K_\X^L(M)$,
	then $\K_\X(M)$ is a ${\rm C}^*$-subalgebra that contains $M$ and $JMJ$ in its multiplier algebra \cite[Proposition 3.5]{DKEP22}. 
Put $\K^{\infty,1}_\X(M)=\overline{\K_\X(M)}^{_{\|\cdot\|_{\infty,1}}}\subset \B(L^2M)$, 
	where $\|T\|_{\infty,1}=\sup_{a,b\in (M)_1}\langle T\hat a, \hat b\rangle$,
	and the small-at-infinity boundary for $M$ relative to $\X$ is given by 
$$
\bS_\X(M)=\{T\in\B(L^2M)\mid [T,x]\in \K_\X^{\infty,1}(M),{\rm\ for\ any\ }x\in M'\}.
$$
When $\X=\K(L^2M)$, we omit $\X$ in the above notations.

The following instance of boundary pieces is extensively used in the following.
Let $M$ be a finite von Neumann algebra and $\{P_i\}_{i\in I}$ a possibly infinite family of von Neumann subalgebras.
Recall from \cite[Lemma 6.12]{DP22}
	that the $M$-boundary piece $\X$ associated with $\{P_i\}_{i\in I}$ is the
	hereditary ${\rm C}^*$-subalgebra of $\B(L^2M)$ 
	generated by $\{x JyJ e_{P_i}\mid i\in I,\  x,y\in M\}$. 
	
These notions also accommodate generalizations to the non-tracial setting and we refer to \cite{DP22} for details.

\subsection{Biexactness}
Following Ozawa \cite{Oz04, BrOz08}, one says a discrete group $\Gamma$ is biexact if $\Gamma\actson \overline\Gamma^s$ is topologically amenable.
The corresponding notion for von Neumann algebras was introduced in \cite{DP22} and we recall it in the following.

Let $M$ be a von Neumann algebra, $\X\subset \B(L^2M)$ an $M$-boundary piece and $\bS_\X(M)$ the corresponding small-at-infinity boundary for $M$.
We may equip $S:=\bS_\X(M)$ with the $M$-topology, which is given by the family of seminorms
$$s^\rho_\omega(x)=\inf \{\rho(a^*a)^{1/2}\|y\|\omega(b^*b)^{1/2} \mid x=a^*yb,\ a,b\in M, y\in S\},$$
where $\rho,\omega$ range over all normal positive functionals on $M$.
By an equivalent characterization \cite[Lemma 3.4]{DP22}, 
	we say a net $\{x_i\}\subset S$ converging to $0$ in the $M$-topology if
	there exists a net of projections $p_i\in M$ such that $p_i\to 1$ strongly
	and $\|p_ix_ip_i\|\to 0$.
We also have the weak $M$-topology on $S$, which is described by 
$$\bS_\X(M)^\sharp=\{\varphi\in \bS_\X(M)^*\mid M\times M\ni (a,b)\mapsto \varphi(aTb)\in \C \text{\ is\ separately\ normal}\ \forall T\in \bS_\X(M)\},$$
i.e.,  $\{x_i\}\subset S$ converges to $0$ in the weak $M$-topology if $\varphi(x_i)\to 0$ for any $\varphi\in \bS_{\X}(M)^\sharp$.

We say $M\subset\bS_\X(M)$ is $M$-nuclear if there exist nets of u.c.p.\ maps 
$\phi_i: M\to \M_{n(i)}(\C)$ and $\psi_i: \M_{n(i)}(\C)\to \bS_\X(M)$
such that $\psi_i\circ \phi_i(x)\to x$ in the $M$-topology of $\bS_\X(M)$,
	or equivalently, in the weak $M$-topology (see \cite[Section 4]{DP22} for details).
	
And we say $M$ is biexact relative to $\X$ if $M\subset \bS_\X(M)$ is $M$-nuclear.
When $\X$ is the $M$-boundary piece associated with 
	a family of von Neumann subalgebras $\{P_i\}_{i\in I}$ of $M$,
	we say $M$ is biexact relative to $\{P_i\}_{i\in I}$,
and when $\X=\K^{\infty,1}(M)$, we simply say $M$ is biexact.

These notions coincide with the corresponding notions of groups if we consider group von Neumann algebras:
a discrete group $\Gamma$ is biexact relative to a family of subgroups $\{\Lambda_i\}_{i\in I}$ if and only if
$L\Gamma$ is biexact relative to $\{L\Lambda_i\}_{i\in I}$ \cite[Theorem 6.2]{DP22}.

\begin{lem}\label{lem: central state from non intertwine}
Let $M$ be a finite von Neumann algebra and $P\subset pMp$, $Q\subset M$ von Neumann subalgebras.
Suppose $P\not\prec_M Q$.
Then there exists a $P'\cap pMp$-central state on $\varphi:\bS_{\X_Q}(M)\to \C$ such that
	$\varphi_{\mid pMp}$ coincides with its trace.
Equivalently, there exists a conditional expectation $\phi:\bS_{\X_Q}(M)\to P'\cap pMp$ such that $\phi_{\mid pMp}$ coincides with the canonical expectation onto $P'\cap pMp$.
\end{lem}

\begin{proof}
From $P\not\prec_M Q$ we obtain a net of unitaries $\{u_n\}\subset \cU(P)$ such that 
	$\|E_Q(au_nb)\|_2\to 0$ for any $a,b\in M$,
	which in turns implies that $\|u_n K u_n^*\|_{\infty,2}\to 0$ for any $K\in \K_{\X_Q}^{\infty,1}(M)$ by the same proof of \cite[Lemma 6.12]{DP22}.
The desired $\varphi\in \bS_{\X_Q}(M)^*$ is then obtained by taking a limit point of $\{\langle \cdot \hat u_n, \hat u_n\rangle/\tau(p)\}_n$.

To see the existence of $\phi$, consider $\phi: \bS_{\X_Q}(M)\to \B(L^2(B))$ given by $\langle \phi(T)\hat a, \hat b\rangle=\varphi(b^* Ta)$
	for any $a,b\in B$, where $B=P'\cap pMp$.
Since $\varphi$ is $B$-central, we have $[\phi(T), JxJ]=0$ for any $x\in B$.
The fact $\phi_{\mid pMp}=E^{pMp}_{P'\cap pMp}$ follows from  $\varphi_{\mid pMp}=\tau$.
\end{proof}

\begin{prop}\label{prop: infinite rel biexact}
Given an index set $I$, let $M_i$ be a von Neumann algebra with a faithful normal state $\varphi_i$ for each $i\in I$.
Set $\cM=\ovt_{i\in I} (M_i,\varphi_i)$ and $M_{\widehat i}=\ovt_{j\in I, j\neq i} (M_j,\varphi_j)$ for each $i\in I$.
If each $M_i$ is biexact, then $\cM$ is biexact relative to $\{M_{\widehat i}\}_{i\in I}$.

In particular, if $\{\Gamma_i\}_{i\in I}$ is a family of biexact groups, then $\oplus_{i\in I}\Gamma_i$ is biexact relative to $\{\oplus_{j\neq i}\Gamma_j\}_{i\in I}$. 
\end{prop}

\begin{proof}
We may assume $I$ is infinite as the finite case is covered by \cite[Proposition 6.14]{DP22}.
For any finite set $F\subset I$, denote by $\cM_F=\ovt_{i\in F}M_i$
	and $E_F: \cM\to \cM_F$ the conditional expectation. 
Note $\cM_F$ is biexact relative to $\{E_F(M_{\widehat i})\}_{i\in F}$ as $E_F(M_{\widehat i})=\ovt_{j\in F\setminus\{i\}} M_j$ 
	and thus we have the following diagram
\[\small\begin{tikzcd}
\cM \arrow[r, "E_F"] & \cM_F \arrow[dr, "\phi_k^F"] & 
	& \bS_{\X_F}(\cM_F) \arrow[hookrightarrow]{r}{\iota} &\bS_{\X}(\cM),\\
	&	& \M_{n(k)}(\C) \arrow[ur, "\psi_k^F"] &	 &
\end{tikzcd}\]	
	where $\X_F$ denotes the $\cM_F$-boundary piece 
	associated with $\{E_F(M_{\widehat i})\}_{i\in F}$,
	$\X$ denotes the $\cM$-boundary piece 
	associated with $\{M_{\widehat i}\}_{i\in I}$,
	$\phi_k^F$, $\psi_k^F$ are u.c.p.\ maps such that 
		$\theta_k^F:=\psi_k^F\circ \phi_k^F\to \id_{\cM_F}$ in the point-$\cM_F$-topology,
	and $\iota$ is the restriction of the embedding 
		$\B(\otimes_{i\in F}L^2(M_i,\varphi_i))\otimes 1\subset \B(\otimes_{i\in I}L^2(M_i,\varphi_i))$
		to $\bS_{\X_F}(\cM_F)$.

For any $\varphi\in \bS_\X(\cM)^\sharp$, 
	note that $\varphi\circ \iota\in \bS_{\X_F}(\cM_F)^\sharp$.
For any $x\in \cM$, one checks that 
$$
\lim_k\langle (\iota\circ \theta_k^F)(E_F(x)),\varphi\rangle
	=\lim_k \langle \theta_k^F(E_F(x)),  \varphi\circ \iota\rangle
	=\langle E_F(x), \varphi\rangle.
$$

Since $\varphi_{\mid \cM}$ is normal, we further have $\lim_F \varphi(E_F(x))=\varphi(x)$. 
It follows that $\cM$ is biexact relative to $\X$.
The moreover part is a consequence of \cite[Theorem 6.2]{DP22}.
\end{proof}

\begin{rem}
Note that one may prove directly, without using \cite{DP22}, that $\Gamma=\oplus_{i\in I}\Gamma_i$ is biexact relative to $\cG:=\{\Gamma_{\widehat i}\}_{i\in I}$ if each $\Gamma_i$ is biexact, following the same idea.
Indeed, for each finite $F\subset I$ denote by $\pi_F: \Gamma\to \times_{i\in F}\Gamma_i=:\Gamma_F$ the quotient map 
	and $\cG_F=\{\pi_F(\Gamma_{\widehat i})\}_{i\in F}$.
Following the notation of \cite[Chapter 15]{BrOz08}, one checks that $c_0(\Gamma_F, \cG_F)\subset c_0(\Gamma, \cG)$ under
	the embedding $\ell^\infty(\Gamma_F)=\ell^\infty(\Gamma_F)\ot
	 1_{\ot _{i\in I\setminus F}} \ell^2\Gamma_i\subset \ell^\infty\Gamma$.
Then a similar argument as above shows that $\Gamma\actson \{f\in \ell^\infty\Gamma\mid f-R_tf\in c_0(\Gamma, \cG)\}$ is topologically ameanble, 
	i.e., $\Gamma$ is biexact relative to $\cG$.
\end{rem}

%In the above setting, $\cM=\cM\ovt R$ as $\cM$ is McDuff.
%\textcolor{blue}{Is it clear that $R\prec_\cM M_{\widehat i}$ for some $i\in I$?}

\begin{cor}[cf.\ {\cite[Theorem 5.1]{HoIs16}}]\label{corollary.relative.solidity}
Given an index set $I$, let $M_i$ be a von Neumann algebra with a faithful normal state $\varphi_i$ for each $i\in I$
	and $M_0$ an amenable von Neumann algebra (possibly trivial).
Set $\cM=\ovt_{i\in I} (M_i,\varphi_i)$ and $M_{\widehat i}=\ovt_{j\in I, j\neq i} (M_j,\varphi_j)$ for each $i\in I$.

If $P\subset M_0\ovt \cM$ is a finite von Neumann subalgebra with expectation,
	then either $P'\cap (M_0\ovt \cM)$ is amenable or $P\prec_{M_0\ovt \cM} M_0\ovt M_{\widehat i}$ for some $i\in I$.
\end{cor}
\begin{proof}
Set $\widetilde \cM=M_0\ovt \cM$ and $\widetilde M_{\widehat i}=M_0\ovt M_{\widehat i}$. 
Suppose $P\not\prec_{\widetilde \cM}\widetilde {M_{\widehat i}}$ for any $i\in I$.
By \cite[Lemma 6.12]{DP22} and its following paragraph, one obtains a u.c.p.\ map $\Theta: \bS_{\widetilde \X}(\widetilde \cM)\to \widetilde\cM$
	such that $\Theta_{\mid Q}=\id_Q$, where $Q=P'\cap \widetilde\cM$
	and $\widetilde\X$ denotes the $\tilde \cM$-boundary piece associated with $\{\widetilde M_{\widehat i}\}_{i\in I}$.

Denote by $\X$ the $\cM$-boundary piece associated with $\{M_{\widehat i}\}_{i\in I}$.
As $\cM$ is biexact relative to $\X$ by Proposition~\ref{prop: infinite rel biexact}, the $\cM$-boundary piece associated with $\{M_{\widehat i}\}_{i\in I}$, 
	we have nets of u.c.p.\ maps $\phi_j: \cM\to \M_{n(j)}(\C)$ and $\psi_j: \M_{n(j)}(\C)\to \bS_\X(\cM)$ such that 
	$\psi_j\circ \phi_j\to \id_\cM$ in the point $\cM$-topology.
Moreover, since $M_0$ is semidiscrete, we also have u.c.p.\ maps $\alpha_\ell : M_0\to \M_{n(\ell)}(\C)$ and $\beta_\ell: \M_{n(\ell)}(\C)\to M_0$
	such that $\beta_\ell\circ \alpha_\ell \to \id_{M_0}$ in the point-weak$^*$ topology.
Consider
\[\small\begin{tikzcd}
M_0\otimes_{\rm min} \cM  \arrow[dr, "\alpha_\ell \otimes \phi_j"] &   & M_0\otimes_{\rm min}\bS_\X(\cM)\subset \bS_{\widetilde \X}(\widetilde \cM). \\
& \M_{n(\ell)}(\C)\otimes\M_{n(j)}(\C) \arrow[ur, "\beta_\ell\otimes \psi_j"] & 
\end{tikzcd}\]
One checks that $(\beta_\ell\otimes \psi_j)\circ(\alpha_\ell\otimes \phi_j)(x)\to x$ for any $x\in M_0\otimes_{\rm min} \cM$
	in the $(M_0\otimes_{\rm min} \cM\subset \widetilde \cM)$-topology, from which we obtain that $\widetilde \cM\subset \bS_{\widetilde X}(\widetilde \cM)$ is $\cM$-nuclear by \cite[Corollary 4.9]{DP22}.

Finally, note that there exists a normal faithful conditional expectation $E: \widetilde \cM\to Q=P'\cap \widetilde \cM$ 
	since $P\subset \widetilde M$ is with expectation.
It follows that 
$$Q\subset \widetilde \cM\subset \bS_{\widetilde X}(\widetilde \cM)\xrightarrow{\Theta} \widetilde \cM\xrightarrow{E} Q$$
is a weakly nuclear map, i.e., $Q$ is semidiscrete.
\end{proof}

As a consequence, we obtain the corresponding result of \cite{HoIs16} in the infinite tensor product setting.
This also removes the weak amenability assumption in \cite{Iso19} \cite[Theorem D]{Iso20}, as we do not rely on strong primeness.

\begin{cor}\label{cor: infinite UPF}
Let $I$ and $J$ be index sets and $(M_i, \varphi_i)$, $(N_j,\psi_j)$ be nonamenable factors for each $i\in I$ and $j\in J$ with faithful normal states.
Set $M=(\ovt_{i\in I}(M_i,\varphi_i))\ovt M_0$ and $N=(\ovt_{j\in J} (N_j, \psi_j))\ovt N_0$, where $M_0$ and $N_0$ are amenable factors.

If each $M_i$ is biexact and $(N_j,\psi_j)$ is prime with large centralizer,
	then $M=N$ implies that
    there exists a bijection $\sigma:J\to I$ such that $M_{\sigma(j)}$ is stably isomorphic to $N_j$ for all $j\in J$.
\end{cor}

\begin{proof}
This is similar to \cite[Corollary 6.15]{DP22}.
Following the exact same proof of \cite[Lemma 5.2]{HoIs16} using Corollary~\ref{corollary.relative.solidity} 
	in replacement of \cite[Theorem 5.1]{HoIs16},
	one has that for any $j\in J$, there exists some $\sigma(j)\in I$ such that $M_{\sigma(j)}\prec_M N_j$.
Combining with primeness of $N_j$, we further have the stable isomorphism between $M_{\sigma(j)}$ and $N_j$.
The fact that $\sigma$ is a bijection follows from the same proof of \cite[Theorem A]{Iso19}.
%To see $\sigma: J\to I$ is injective, suppose there exist $j_1, j_2\in J$ with $\sigma(j_1)=\sigma(j_2)=i_0$.
%As $M_{i_0}\prec_M N_{j_1}$,	
%	there exist nonzero projections $p_1\in M_{i_0}$, $p_0\in M_{\widehat {i_0}}$,
%	$q_1\in N_{j_1}$, $q_0\in N_{\widehat {j_1}}$ and a partial isometry $v\in \cM$ with $vv^*=q_1q_0=:q$ and $v^*v=p_1p_0=:p$
%	such that $vM_{i_0}v^*\subset qN_{j_1}q$ is with expectation \cite[Lemma 4.13]{HoIs16}.
%By $M_{i_0}\prec_{\cM} N_{j_2}$ and factoriality of $M_{i_0}$ and $M_{\widehat {i_0}}$,
%	we have $v^*v M_{i_0} v^*v\prec_{\cM} N_{j_2}$ as well \cite[Remark 4.2, 4.5]{HoIs16}.
%It follows from \cite[Lemma 4.4]{Iso19} that $j_1=j_2$.
\end{proof}

\section{Relative solidity in measure equivalence}\label{sec: rel solid}

This section is devoted to prove Theorem~\ref{thm: rel solid} and its more technical versions, which are needed for later use.
We first show the following measure equivalence version of Theorem~\ref{thm: rel solid}.

\begin{thm}\label{thm: rel solid in ME}
Let $\Gamma$ and $\Lambda=\Lambda_1\times \Lambda_2$ be countable groups 
	with $\Gamma\sim_{\rm ME} \Lambda$ via a coupling $\Omega$.
Suppose $\Lambda_1$ is nonamenable and biexact, and $\Sigma<\Gamma$ is subgroup. 
If $\Sigma\not\preccurlyeq_\Omega \Lambda_2$, 
	then $C_\Gamma(\Sigma)$ is amenable relative to $\Lambda_2$ via $\Omega$.
	
Moreover, denote by $\Gamma\actson (X,\mu)$ and $\Lambda\times \Z/d\Z\actson (Y\times \Z/d\Z,\nu\times c)$ 
	the associated stably orbit equivalent free ergodic m.p.\ actions
	with $(X,\mu)$ realized as a measurable subset of $(Y\times \Z/d\Z, \nu\times c)$,
	and set $B=L^\infty(Y\times \Z/d\Z,\nu\times c)$ and $M=B\rtimes(\Lambda\times \Z/d\Z)$.
We then have either $z L(\Sigma)\prec_M B\rtimes\Lambda_2$ 
	or $z L(C_\Gamma(\Sigma))$ is amenable relative to $B\rtimes\Lambda_2$ in $M$,
	for any nonzero $\Sigma C_{\Gamma}(\Sigma)$-invariant projection $z\in L^\infty(X)$.

\end{thm}
\begin{proof}
Replacing $\Lambda_1$ with $\Lambda_1\times \Z/d\Z$ and setting $Y_1=Y\times \Z/d\Z$,
	we have $\cR({\Gamma\actson X})=\cR({\Lambda\actson Y_1})\cap (X\times X)$.
Denote by $p:=1_X\in B= L^\infty(Y_1)$, $A:=L^\infty(X)=pB$, $M=B\rtimes\Lambda$,
	$\{u_t\mid t\in \Gamma\}\subset A\rtimes\Gamma$ and $\{v_g\mid g\in \Lambda\}\subset M$ the canonical unitaries
	and $c: \Gamma\actson X\to \Lambda$ the Zimmer cocycle.

We will only show the moreover part, as the first assertion follows by setting $z=1_X$ and using Lemma~\ref{lem: ME intertwining}.
View $L\Gamma$ as a von Neumann subalgebra of $pMp$ and set $\Delta: M\to M\ovt L\Lambda$ to be the comultiplication map
	given by $\Delta(bu_g)=bu_g\ot u_g$ for $b\in B$ and $g\in \Lambda$.
Since $zL\Sigma\not\prec_M B\rtimes\Lambda_2$,
	one has $(z\ot 1)\Delta(L\Sigma)\not\prec_{M\ovt L\Lambda} M\ovt L\Lambda_2$.
Lemma~\ref{lem: central state from non intertwine} yields conditional expectation $E:\bS_{\X_{M\ovt L\Lambda_2}}(M\ovt L\Lambda)\to P$
	with $E_{\mid M\ovt L\Lambda}$ coinciding with the canonical conditional expectation from $M\ovt L\Lambda$ to $P$,
	where $P=\Delta(zL\Sigma)'\cap (zMz\ovt L\Lambda)$ contains $\Delta(zL(C_\Gamma( \Sigma))$.

We follow the idea of \cite[Proposition 3.1]{Din24}.
As $\Lambda_1$ is biexact, there exists sequences of u.c.p.\ maps
$\alpha_n: \B(\ell^2\Lambda_1)\to \M_{k(n)}(\C)$ and $\beta_n: \M_{k(n)}(\C)\to \bS(L\Lambda_1)$ such that 
$\beta_n\circ \alpha_n\to \id_{L\Lambda_1}$ in the $L\Lambda_1$-topology. 
Moreover, we may assume $\alpha_n$ is normal \cite[Lemma 4.1]{DP22}.
It follows that we have
$$
\phi_n= : (M\ovt L\Lambda_2)\ovt \B(\ell^2\Lambda_1)\to (M\ovt L\Lambda_2)\ot_{\rm min} \bS(L\Lambda_1)\subset \bS_{\X_{M\ovt L \Lambda_2}}(M\ovt L\Lambda),
$$
where $\phi_n=(\id_{M\ovt L\Lambda_2}\ot \beta_n)\circ (\id_{M\ovt L\Lambda_2}\ot \alpha_n)$,
such that $\phi_n {\mid_{ M\ovt L\Lambda_2}}=\id$ and $\phi_n(x)\to x$ in the $M\ovt L\Lambda$-topology for any $x\in L\Lambda_1$,
	which implies that $\phi_n(x)\to x$ in $M\ovt L\Lambda$-topology for any $x\in (M\ovt L\Lambda_2)\ot_{\rm alg} L\Lambda_1$
	(N.B. we do not conclude convergence for all $x\in (M\ovt L\Lambda_2)\ovt L\Lambda_1$).
Here, we have $(M\ovt L\Lambda_2)\ot_{\rm min} \bS(L\Lambda_1)\subset \bS_{\X_{M\ovt L \Lambda_2}}(M\ovt L\Lambda)$
	as $(M\ovt L\Lambda_2)\ot \K(L\Lambda_1)\subset \K_{\X_{M\ovt L\Lambda_2}}(M\ovt L\Lambda)$. 

Using the embedding of $L\Gamma\subset pMp$ is given by $u_t\mapsto \sum_{g\in \Lambda} v_g p_g^t$,
	where $p_g^t\in L^\infty(X)$ is the characteristic function of $\{x\in X\mid c(t,x)=g\}$,
	we may produce a sequence of projections $\{p_n\}\in L^\infty(X)$ that increases to $p$ strongly
	such that for any $t\in \Gamma$, there exists some $N_t\in \N$ with $p_n u_t p_n\in B\rtimes_{\rm alg} \Lambda$ for all $n\geq N_t$.	
Indeed, for any $t\in \Gamma$ and $\varepsilon>0$, we may find a finite subset $F_{t,\varepsilon}\subset \Lambda$ 
	such that $\tau(p-\sum_{g\in F_{t,\varepsilon}} p_g^t)<\varepsilon$
	and we set $q_{t,\varepsilon}=\sum_{g\in F_{t,\varepsilon}} p_g^t$.
After enumerating elements in $\Gamma$ by $\{t_n\}_{n\in \N}$, we set $p_{n}=\wedge_{i\leq n} q_{t_i, 2^	{-n}}$.
Note that $\tau(p-p_n)\leq n2^{-n}$ and for any $i\leq n$, 
	we have $p_n u_{t_i} p_n=p_n(\sum_{g\in F_{t_i, 2^{-n}}}v_gp_g^t)p_n\in B\rtimes_{\rm alg}\Lambda$.

Set $\psi_n^0: pMp\to pMp$ to be $\psi_n^0(x)=p_n x p_n+\tau(x)(p-p_n)/\tau(p)$
	and $\psi_n=\psi_n^0\otimes\id_{B(\ell^2\Lambda_1)\ovt L\Lambda_2}$.
Then for any $t\in \Gamma\setminus\{e\}$ and $n\geq N_t$, we have
	$\psi_n(\Delta(u_t))=\sum_{g\in \Lambda} \Ad(p_n)( v_g p_g^t)\ot v_g=\sum_{g\in F} \Ad(p_n)( v_g  p_g^t)\ot v_g\in M\ot_{\rm alg} L\Lambda$,
	for some finite subset $F\subset \Lambda$ depending on $n$ and $t$.	
	Thus it follows that for any $t\in \Gamma$ and $m\geq N_t$,
	we have $\phi_n(\psi_m(\Delta(z u_t)))\to \psi_m(\Delta(z u_t))$ in the $M\ovt L\Lambda$-topology
	as $\psi_m(\Delta(z u_t))\in (M\ovt L\Lambda_2)\ot_{\rm alg} L\Lambda_1$. 
Next, notice that $\psi_m\to \id_{pMp\ovt L\Lambda}$ in the $pMp\ovt \Lambda$-topology 
	as $s^{\rho}_\omega(p_nx p_n-x)\leq \rho(p-p_n)+\omega(p-p_n)$ for any $x\in (pMp\ovt L\Lambda)_1$ 
	and normal states $\rho,\omega\in (pMp\ovt L\Lambda)_{*}$.
We also observe that $E: \bS_{\X_{M\ovt L\Lambda_2}}(M\ovt L\Lambda)\to P$ is continuous from the weak $pMp\ovt L\Lambda$-topology 
	to the weak $P$-topology as $E_{\mid pMp\ovt L\Lambda}$ is the conditional expectation to $P$.

Now consider the u.c.p.\ map $\theta_{n,m}:=E\circ \phi_n\circ \psi_m: pMp\ovt L\Lambda_2\ovt \B(\ell^2\Lambda_1)\to P$.
For any $t\in C_\Gamma({\Sigma})$, the above argument shows that 
	$\lim_m\lim_n E(\phi_n(\psi_m(\Delta(z u_t))))=\lim_m E(\psi_m(\Delta(z u_t))=E(\Delta(z u_t))=\Delta(z u_t)$
	in the ultraweak topology of $P$ (equivalently, the weak $P$-topology of $P$).
Moreover, for any $a\in A$, we have
	$E(\phi_n(\psi_m(a)))=E(\psi_m(a))\to E(a)$ ultraweakly in $P$. 
	
Therefore, if we denote by $\theta$ a limit point of $\{\theta_{n,m}\}$ in the point-weak$^*$ topology,
	then we obtain a state $\tau_P\circ \theta: pMp\ovt \B(\ell^2\Lambda_1)\ovt L\Lambda_2\to \C$
	that is $\{\Delta(zu_t)\mid t\in C_{\Gamma}(\Sigma)\}$-central and $(\tau\circ \theta)_{\mid zB\ot 1}$
	coincides with its trace.
Regarding $M\ovt \B(\ell^2\Lambda_1)\ovt L\Lambda_2=
	\langle B\rtimes(\Lambda \times \Lambda), e_{B\rtimes(\Lambda\times \Lambda_2)}\rangle$,
	where $\{e\}\times \Lambda$ acts trivially on $B$,
	we may apply Lemma~\ref{lem: central state} to conclude that 
	$\Delta(zL(C_\Gamma(\Sigma)))$ is amenable relative to $M\ovt L\Lambda_2$ in $M\ovt L\Lambda$,
	which implies that $zL(C_\Gamma(\Sigma))$ is amenable relative to $B\rtimes\Lambda_2$ in $M$
	by Lemma~\ref{lem: commultiplication and rel amen}.
\end{proof}

\begin{rem}\label{rem: technical}
Note that from the biexactness assumption, we derive a sequence of u.c.p.\ maps $\phi_n$	
	that only converges to identity on $(M\ovt L\Lambda_2)\ot_{\rm alg} L\Lambda_1$.
This is exactly the technical issue in \cite[Proposition 3.1]{Din24} and \cite[Proposition 7.3]{Iso20},
	which was overcame by the extra assumption of weak amenability.
However, the measure equivalence setting here allows us to construct approximation maps $\psi_n$ from above to circumvent this issue.
\end{rem}

\begin{cor}\label{cor: fundamental group}
    Let $\Gamma\actson (X,\mu)$ and $\Lambda\actson (Y,\nu)$ be p.m.p.\ free ergodic actions of countable groups
	and denote by $\cR=\cR(\Gamma\actson X)$, $\cS=\cR(\Lambda\actson Y)$.
Suppose $L(\cR)$ and $L(\cS)$ are full.

If $\Gamma$ is biexact, then $\cF(\cR\times \cS)=\cF(\cR)\cF(\cS)$.
\end{cor}

\begin{proof}
Suppose $t\in \cF(\cR\times \cS)$ and we have $\cR\times \cS=(\cR\times \cS)^t=\cR\times \cS^t=: \cK\times \cL$,
	and we further denote $\cK=L^\infty(X_1)\rtimes\Gamma_1$, $\cL=L^\infty(Y_1)\rtimes\Lambda_1$
	and by $\Omega$ the ME coupling between $\Gamma\times \Lambda$ and $\Gamma_1\times \Lambda_1$.
Following \cite{Iso20}, we claim that one has $\Lambda\preccurlyeq_\Omega \Lambda_1$, 
	or $\Gamma_1\preccurlyeq_\Omega \Lambda$.
Indeed, if $\Lambda \not\preccurlyeq_\Omega \Lambda_1$ and $\Gamma_1\not\preccurlyeq_\Omega \Lambda$, 
	we then have $L\Gamma$ is amenable relative to $A\rtimes\Lambda_1$,
	as well as $A\rtimes\Lambda_1$ is amenable relative to $A\rtimes\Lambda$
	by Theorem~\ref{thm: rel solid in ME} and Lemma~\ref{lem: ME intertwining},
	where $A=L^\infty(X\times Y)=L^\infty(X_1\times Y_1)$.
By \cite{OzPo10I}, we have $L\Gamma$ is amenable relative to $A\rtimes\Lambda$ in $A\rtimes(\Gamma\times \Lambda)=:M$, which implies $\Gamma$ is amenable.

In the case of $\Lambda\preccurlyeq_{\Omega} \Lambda_1$, we have $1\ot L(\cS)\prec_M A\rtimes \Lambda_1$  by Lemma~\ref{lem: ME intertwining}.
Using the fact that $L(\cS)$ is full, we have $1\ot L(\cS)\prec_M 1\ot L(\cL)$ by \cite[Proposition 6.3]{Hof16}
	and hence $u^*(1\ot L(\cS)) u\subset 1\ot L(\cL)^s\subset L(\cK)^{1/s}\ovt L(\cL)^s$
	for some $s>0$ and $u\in \cU(M)$.
It follows that $L(\cL)^s=u^*(1\ot L(\cS))u\ovt P$ and $u^*(L(\cR)\ot 1)u=P\ovt L(\cK)^{1/s}$,
	where $P=(u^* (1\ot L(\cS))u)'\cap L(\cL)^s$.
Since $L(\cR)$ is solid \cite{Oz06}, one has $P=\M_n(\C)$ for some $n\in \N$.

Lastly, notice from above one has $u^*(1\ot L^\infty(Y))u\prec_{1\ot L(\cL)^s} L^\infty(Y_1)^s$ 
	and $L^\infty(X_1)^{1/s}\prec_{u^*(L(\cR)\ot 1)u}u^*( L^\infty(X)\ot 1)u$.
It then follows from \cite{Po06A} that $\cR=\cK^{n/s}=\cR^{n/s}$ and $\cS^n=	\cL^s=\cS^{ts}$,
	i.e, $n/s\in \cF(\cR)$ and $ts/n\in \cF(\cS)$, as desired.
	
If $\Gamma_1\preccurlyeq_\Omega \Lambda$, by a similar argument we have $v^*(L(\cK)\ot 1)v\subset 1\ot L(\cS)^k$
	for some $k>0$ and $v\in \cU(M)$
	and hence $1\ot L(\cS)^k=v^*(L(\cK)\ot 1)v\ovt Q$ and $v^* (1\ot L(\cL))v=L(\cR)^{1/k}\ovt Q$,
	where $Q=(v^* (L(\cK)\ot q)v)'\cap (1\ot L(\cS)^k)=\M_m(\C)$ for some $m\in \N$
	by fullness of $L(\cL)$.
The same argument as above then shows $\cS^k=\cK^m=\cR^m$ and $\cS^t=\cL=\cR^{m/k}$ and hence $t\in \cF(\cS)$.
\end{proof}

\subsection{Technical variants}
In this section, we prove technical variants of Theorem~\ref{thm: rel solid in ME} for later use.
We will assume the following notation throughout.

\begin{note}\label{notation}
Let $\Lambda$ be a countable i.c.c.\ group that is measure equivalent to a group $\Gamma$. 
By using \cite[Lemma 3.2]{Fu99A}, there exist positive integers $d\leq \ell $, 
	free ergodic p.m.p.\ actions $\Gamma\car (X,\mu)$ and $\Lambda\car (Y,\nu)$ such that 
$$
\mathcal R(\Lambda\actson Y)=\mathcal R(\Gamma\times \mathbb Z/d\mathbb Z \car X \times \mathbb Z/d\mathbb Z)\cap (Y\times Y),
$$
$$
\cR(\Gamma\times \Z/d\Z\actson X\times \Z/d\Z)= \cR(\Lambda\times \Z/\ell \Z\actson Y\times \Z/\ell \Z)\cap ((X\times \Z/d\Z)\times (X\times \Z/d\Z)).
$$

Here, we considered that $\mathbb Z/d\mathbb Z \car (\mathbb Z/d\mathbb Z,c)$ acts by addition and $c$ is the counting measure and similarly for $\ell$.

We also identified $X\times \mathbb Z/d\mathbb Z$ as a measurable subset of $Y\times \mathbb Z/\ell \mathbb Z$.
	$Y=Y\times \{0\}$ as a measurable subset of $X\times \Z/d\Z$. 
Denote $p=1_Y \in L^\infty(X\times \mathbb Z/d\mathbb Z)$ and $q=1_X\in L^\infty(Y\times \Z/\ell\Z)$.

Letting $B=L^\infty(Y)$, $A=L^\infty(X)\otimes \M_d(\mathbb C)$, $B_1=L^\infty(Y)\ot \M_\ell(\C)$ and $M=A\rtimes\Gamma$, 
	we have $pMp=B\rtimes\Lambda$ and $B\subset pAp$, as well as
	$M=q(B_1\rtimes\Lambda)q$ and $A\subset qB_1 q$.
Denote by $\{u_g\}_{g\in \Gamma}$ and $\{v_\lambda\}_{\lambda\in \Lambda}$ the canonical unitaries implementing the actions $\Gamma\car A$ and $\Lambda\car B$, respectively.

Following \cite{PoVa10a} we define the $*$-homomorphism $\Delta: B_1\rtimes\Lambda \to (B_1\rtimes\Lambda) \bar\otimes L(\Lambda)$ by $\Delta(bv_\lambda)=bv_\lambda\otimes v_\lambda$, for all $b\in B_1,\lambda\in\Lambda$. 
We may then restrict $\Delta$ to a $*$-homomorphism $\Delta:q(B_1\rtimes \Lambda)q=M\to M\bar\otimes L(\Lambda)$ and verify that  $\Delta (M)'\cap M\bar\otimes L(\Lambda)=\mathbb C 1$ since $\Lambda$ is i.c.c.
\end{note}

%\textcolor{cyan}{I made slight change to the above setup so that we can prove the desired form of 3.9 and 3.10 easier.}

\begin{thm}\label{cor: rel solid in ME}
With Notation~\ref{notation}, suppose $\Gamma=\Gamma_1\times \Gamma_2$ with $\Gamma_1$ nonamenable and biexact.
Then for any $\Sigma<\Gamma$ and any nonzero $\Delta(\Sigma C_\Gamma(\Sigma))$-invariant projection $z\in L^\infty(X\times X)$
	we have either $z\Delta(L(\Sigma))\prec_{M\ovt M} M\ovt L^\infty(X)\rtimes\Gamma_2$
    or $z\Delta(L(C_\Gamma(\Sigma))$ is amenable relative to $M\ovt (L^\infty(X)\rtimes\Gamma_2)$ in $M\ovt M$.
\end{thm}
\begin{proof}
Consider $\Psi^0: A\rtimes\Gamma\ni au_t\mapsto au_t\ot u_t\in  M\ovt L\Gamma$ and its extension
	$$\Psi:=\id_ M\ot \Psi^0: M\ovt (A\rtimes\Gamma)\to M\ovt (M\ovt L\Gamma).$$
Denote by $\cM=M\ovt M\ovt L\Gamma$ and $\X$ the $\cM$-boundary piece associated with $M\ovt M\ovt L\Gamma_2$.
Since $\Gamma_1$ is biexact, we have a sequence of u.c.p.\ maps
$$\phi_n: \cP:= M\ovt M\ovt  \B(\ell^2\Gamma_1)\ovt L\Gamma_2\to \bS_{\X}(\cM)$$
with ${\phi_n}_{\mid M\ovt M\ovt L\Gamma_2}=\id$ and $\phi_n(x)\to x$ in the $\cM$-topology for any $x\in L\Gamma_1$.

Set $N=z\Delta(L\Sigma)$ and $N_1=\Psi(z\Delta(L\Sigma))$.
Assuming $N\not\prec_{M\ovt M} M\ovt (A\rtimes\Gamma_2)$, we have $N_1\not\prec_{\cM} M\ovt M\ovt L\Gamma_2$ and hence
	there exists an $N_2$-central state $\varphi: \bS_{\X}(\cM)\to \C$
	such that $\varphi_{\mid \tilde z \cM\tilde z}=\tau$ by Lemma~\ref{lem: central state from non intertwine},
	where $N_2=N_1'\cap \tilde z \cM \tilde z$ and $\tilde z=\Psi(z)=z\ot 1_{L\Gamma}$.
	
We claim that we may find a sequence of u.c.p.\ maps
	$\theta_i: \cP\to \cP$ 
	such that for any $x\in \Psi\big(z\Delta(\C\Gamma)\big)$ we have
	$\theta_i(x)\in (M\ovt M)\ot_{\alg}L\Gamma$ for large enough $i$ 
	and $\theta_i(x)\to x$ in the point $\cM$-topology, similar to the argument in Theorem~\ref{thm: rel solid in ME}.
Proceeds as in Theorem~\ref{thm: rel solid in ME}, we then obtain a $\Psi(\Delta(\Sigma^c))$-invariant state $\psi$ on $\cP$
	that restricts to the trace on $z L^\infty(X\times X)$,
	where $\Sigma^c=C_\Gamma(\Sigma)$.
It then follows from Lemma~\ref{lem: central state} that $\Psi(z\Delta(L\Sigma^c))$ 
	is amenable relative to $M\ovt M\ovt L\Gamma_2$ in $\cM$, which in turns implies
	$z\Delta(L\Sigma^c)$ is amenable relative to $M\ovt (A\rtimes\Gamma_2)$ in $M\ovt M$
	by Lemma~\ref{lem: commultiplication and rel amen}.

To see the existence of such $\theta_i$, we first analyze $z\Delta(u_t)\in M\ovt M$ for $t\in \Gamma$.
Set $X_1=X\times \Z/d\Z$, $\tilde \Gamma=\Gamma\times \Z/d\Z$, $Y_1=Y\times \Z/\ell\Z$ and $\tilde \Lambda=\Lambda\times \Z/\ell\Z$.

For each $t\in \Gamma$, we have $u_t=\sum_{g\in \tilde \Lambda}q_g^t v_g$, where $q_g^t\in L^\infty(Y_1)$ forms a partition of $q=1_{X_1}\in L^\infty(Y_1)$.	
Similarly, we have $v_g =\sum_{t\in \tilde \Gamma} p_t^g u_t$ for each $g\in \Lambda$, 
	where $p_t^g\in L^\infty(Y)$ forms a partition of $p=1_Y\in L^\infty(X_1)$.
It follows that
$$
\Delta(u_t)=\sum_{(g,h)\in \tilde \Lambda} q_{(g,h)}^t v_{(g,h)}\ot v_g
$$
and 
$$
\Psi(\Delta(u_t))=\sum_{(g,h)\in \tilde \Lambda}\sum_{(s,r)\in \tilde \Gamma} q_{(g,h)}^tv_g\ot p_{(s,r)}^g u_{(s,r)}\ot u_s.
$$
The same argument as in Theorem~\ref{thm: rel solid in ME} produces two sequences of projections $\{q_n\}, \{p_m\}\subset L^\infty(X_1)$
	strongly increasing to $q$ and $p$, respectively.
And for each $t\in \Gamma$ (resp.\ $g\in \Lambda$), there exists $N_t\in \N$ (resp.\ $M_g\in \N$)
	such that $q_n\sum_{g\in \Lambda_1} q_g^t=q_n\sum_{g\in F}q_g^t$ (resp.\ $p_m\sum_{s\in \Gamma} p_s^g=p_m\sum_{s\in E}p_s^g$)
	for any $n\geq N_t$ (resp.\ $m\geq M_g$), where $F\subset \Lambda_1$ and $E\subset \Gamma_1$ are some finite sets.

Now we view $p_n\in L^\infty(X_1)\rtimes \tilde \Gamma= M$
	and consider $\alpha_n(q)=\Ad(q_n)(x)+\tau_{M}(x)q_n^\perp$ as a u.c.p.\ map on $M$.
Similarly, we consider $\beta_m^0(x)=\Ad(p_m)(x)+\tau_M(x)p_m^\perp$ for $x\in M$
	and set $\beta_m=\beta_m^0\ot \id_{\B(\ell^2\Gamma_1)\ovt L\Gamma_2}$.
Put $\theta_{n,m}:=\alpha_n\ot \beta_m \in UCP( \cP)$ and enumerate $\Gamma=\{t_i\}_{i\in \N}$.

For each $i\in \N$, set $n_i:=\max_{j\leq i}\{N_{t_j}\}$, $F_i\subset \Lambda_1$ to be a finite such that
	$q_{n_i}\sum_{g\in \Lambda_1} q_g^{t_j}=q_{n_i}\sum_{g\in F_i}q_g^{t_j}$ for all $j\leq i$,
	and $m_i:=\max_{g\in \pi(F_i) }\{M_{g}\}$, where $\pi:\Lambda_1\to \Lambda$ is the quotient.
It follows that 
$$\theta_i(\Psi(\Delta(u_{t_j})))=\sum_{(g,h)\in F_i} \sum_{{(s,r)}\in E} q_{n_i}( q_{(g,h)}^t v_{(g,h)}) q_{n_i}\ot p_{m_i}(p_{(s,r)}^g u_{(s,r)})p_{m_i}\ot u_s \in M\ot_{\rm alg} M\otimes_{\rm alg} L\Gamma$$ 

for all $j\leq i$,
	where $\theta_i:=\theta_{(n_i, m_i)}$ and $E\subset \Gamma_1$ is some finite subset.
Finally, since $z\in L^\infty(X\times X)$ is in the multiplicative domain of $\theta_i\circ \Psi$, we have 
	$$
	\theta_i(\Psi(z\Delta(u_{t_j})))= (z\ot 1_{L\Gamma}) \theta_i(\Psi(\Delta(u_{t_j}))) \in (M\ovt M)\otimes_{\rm alg} L\Gamma,
	$$
	as desired.
\end{proof}

The following is another variation of Theorem~\ref{thm: rel solid in ME}. Although it is more involved, 
	the proof follows the same ideas
	as in  Theorem~\ref{thm: rel solid in ME} and Theorem~\ref{cor: rel solid in ME}.

% \begin{lem}\label{lem: flip-rel solid}
% With Notation~\ref{notation}, we assume in addition that $\Lambda=\Lambda_1\times \Lambda_2$.
% Set $\cM=M\ovt L\Lambda=M\ovt (L\Lambda_1\ovt L\Lambda_2)$, $\tilde \cM=M\ovt M$ and $\sigma\in \Aut(\cM\ovt \cM)$ given by
% 	$\sigma((x\ot v_g\ot v_h)\ot (y\ot v_{g'}\ot v_{h'}))=(x\ot v_{g'}\ot v_h)\ot (y\ot v_g \ot v_{h'})$,
% 	where $x,y\in M$, $g,g'\in \Lambda_1$ and $h, h'\in \Lambda_2$.
	
% Suppose that $\Gamma=\Gamma_1\times\Gamma_2$ with $\Gamma_1$ nonamenable and biexact,
% 	$\Sigma<\Gamma$ is a subgroup and $\Sigma^c$ denotes the centralizer of $\Sigma$ in $\Gamma$.

% Then for any nonzero projection $z\in L^\infty(X)$ that is $\Sigma$ and $\Sigma^c$-invariant,
% 		we have either $\sigma(z\Delta(L\Sigma )\ot \Delta({L\Gamma}))\prec_{\cM\ovt \tilde \cM} \cM\ovt (M\ovt A\rtimes\Gamma_2)$
% 		or $\sigma(z\Delta(L\Sigma^c)\ot 1)$ is amenable relative to $\cM\ovt (M\ovt A\rtimes\Gamma_2)$ in $\cM\ovt \tilde \cM$.

% Similarly, we have 
% 	either $\sigma(\Delta({L\Gamma})\ot z\Delta(L\Sigma ))\prec_{\cM\ovt \tilde \cM} \cM\ovt (M\ovt A\rtimes\Gamma_2)$
% 		or $\sigma(1\ot z\Delta(L\Sigma^c ))$ is amenable relative to $\cM\ovt (M\ovt A\rtimes\Gamma_2)$ in $\cM\ovt \tilde \cM$.
% \end{lem}
\begin{thm}\label{lem: flip-rel solid}
With Notation~\ref{notation}, we assume in addition that $\Lambda=\Lambda_1\times \Lambda_2$.
Set $\cM=M\ovt L\Lambda=M\ovt (L\Lambda_1\ovt L\Lambda_2)$, $\tilde \cM=M\ovt M$ and $\sigma\in \Aut(\cM\ovt \cM)$ given by
	$\sigma((x\ot v_g\ot v_h)\ot (y\ot v_{g'}\ot v_{h'}))=(x\ot v_{g'}\ot v_h)\ot (y\ot v_g \ot v_{h'})$,
	where $x,y\in M$, $g,g'\in \Lambda_1$ and $h, h'\in \Lambda_2$.
	
Suppose that $\Gamma=\Gamma_1\times\Gamma_2$ with $\Gamma_1$ nonamenable and biexact,
	$\Sigma<\Gamma$ is a subgroup and $\Sigma^c$ denotes the centralizer of $\Sigma$ in $\Gamma$.

Then we have either $\sigma\big(\Delta(L(\Sigma) )\ovt \Delta({L(\Gamma)})\big)z\prec_{\cM\ovt \tilde \cM}  \cM\ovt M\ovt (A\rtimes\Gamma_2)$
		or $\sigma\big((\Delta(L(\Sigma^c))\ot 1)\big)z$ is amenable relative to $ \cM\ovt M\ovt (A\rtimes\Gamma_2)$ in $ \cM\ovt \tilde \cM$,
	for any nonzero projection $z\in (A\ot 1_{L(\Lambda)})\ovt (A\ovt A)\subset \cM\ovt \tilde \cM$ that is $\sigma\big(\Delta(\Sigma\Sigma^c)\ot \Delta(\Gamma)\big)$-invariant.

Similarly, we have 
	either $\sigma\big(\Delta({L(\Gamma)})\ot \Delta(L(\Sigma) )\big)z\prec_{\cM\ovt \tilde \cM}  \cM\ovt (M\ovt A\rtimes\Gamma_2)$
		or $\sigma\big(1\ot \Delta(L(\Sigma^c))\big)z$ is amenable relative to $ \cM\ovt (M\ovt A\rtimes\Gamma_2)$ in $ \cM\ovt \tilde \cM$,
	for any nonzero projection $z\in (A\ot 1_{L(\Lambda)})\ovt (A\ovt A)\subset \cM\ovt \tilde \cM$ that is $\sigma\big(\Delta(\Sigma\Sigma^c)\ot \Delta(\Gamma)\big)$-invariant.
\end{thm}

\begin{proof}
Consider $\Theta_0: A\rtimes\Gamma\ni au_t\mapsto au_t\ot u_t\in  M\ovt L\Gamma$ and its extension
	$$\Theta:=\id_{\tilde \cM\ovt M}\ot \Theta_0: (\tilde \cM\ovt M)\ovt (A\rtimes\Gamma)\to (\tilde \cM\ovt M)\ovt (M\ovt L\Gamma).$$
Denote by $\cM_1= \tilde\cM\ovt M$, $\cM_2=\cM_1\ovt M\ovt L\Gamma$ and $\X$ the $\cM_2$-boundary piece associated with the subalgebra $\cM_1\ovt M\ovt L\Gamma_2$.
	Since $\Gamma_1$ is biexact, we have a sequence of u.c.p.\ maps
$$\phi_n: \cP:=\cM_1\ovt M\ovt L\Gamma_2\ovt \B(\ell^2\Gamma_1)\to \bS_{\X}(\cM_2)$$
with ${\phi_n}_{\mid \cM_1\ovt M\ovt L\Gamma_2}=\id$ and $\phi_n(x)\to x$ in the $\cM_2$-topology for any $x\in L\Gamma_1$.
In particular, we have $\phi_n(x)\to x$ in the $\cM_2$-topology for any $x\in (\cM_1\ovt M\ovt L\Gamma_2)\ot_{\rm alg} L\Gamma_1$.

Set $N=\sigma\big((\Delta(L\Sigma)\ovt \Delta(L\Gamma)\big)z\subset z( \cM \ovt \tilde \cM)z$ and $N_1=\Theta(N)\subset z\cM_2 z$, where we view $z=z\ot 1_{L\Gamma}\in \cM_2$.
By assumption we have $N\not\prec_{ \cM\ovt \tilde \cM}\cM_1\ovt (A\rtimes\Gamma_2)$, then we have $N_1\not\prec_{\cM_2}\cM_1\ovt M\ovt L\Gamma_2$. 
It follows that there exists a $N_2$-central state $\varphi: \bS_\X(\cM_2)\to \C$ with $\varphi_{\mid z \cM_2 z}=\tau$,
	where $N_2=N_1'\cap z \cM_2 z$.
	
We claim that we may find a sequence of u.c.p.\ maps
	$\theta_i: \cP\to \cP$ 
	such that for any $x\in \Theta\big(\sigma\big( \Delta(\C\Gamma)\otimes 1)  \big)z\big)$ we have
	$\theta_m(x)\in (\cM_1\ovt M)\ot_{\alg}\C\Gamma$ for large enough $m$ 
	and $\theta_i\to \id_{\cM_2}$ in the point $\cM_2$-topology.
Similar to the proof of Theorem~\ref{thm: rel solid in ME},
	this then yields a $\Theta(\sigma\big((\Delta(\Sigma^c)\ot 1)\big)z \big)$-central state $\psi$ on $\cP$
	that restrict to the trace on $z(\cM \ovt \tilde \cM)z$.
We may view $\cM_2=(A\ovt 1\ovt A\ovt A)\rtimes((\Gamma\times \Lambda\times \Gamma\times \Gamma)\times \Gamma)$, 
	where $\{e\}\times \Gamma$ acts trivially,
	and 
$$
\cP=\langle \cM_2, e_{\cM_1\ovt M\ovt L\Gamma_1}\rangle=\langle \mathcal A\rtimes (\Upsilon\times \Gamma) , e_{\mathcal A\rtimes(\Upsilon\times \Gamma_1) }\rangle,
$$
where $\mathcal A=A\ovt 1\ovt A\ovt A$ and $\Upsilon=\Gamma\times \Lambda\times \Gamma\times \Gamma$.
Thus we are in the situation of Lemma~\ref{lem: central state} and hence ${\Theta\big(\sigma\big(\Delta(L\Sigma^c)\ot 1\big)z \big)}$ 
	is amenable relative to $\cM_1\ovt M\ovt L\Gamma_2$ in $\cP$, which in turns implies
	$\sigma(\Delta(L\Sigma^c)\ot 1)z$ is amenable relative to $M_1\ovt (A\rtimes\Gamma_2)$ in $\tilde \cM\ovt \tilde \cM$.
	
To see the existence of such $\theta_i$, we first analyze $\sigma(\Delta(u_t)\ot 1))\in \cM\ovt \cM$ for $t\in \Gamma$.
For each $t\in \Gamma$, we have $u_t=\sum_{g\in \Lambda}p_g^t v_g$, where $p_g^t\in B_1$ forms a partition of $q=1_X$.
Similarly, we have $v_g=\sum_{t\in \Gamma} q_t^g u_t$ for each $g\in \Lambda$, where $\{q_t^g\}_{t\in \Gamma}\subset A$
	forms a partition of $p=1_Y$.
Writing $g=(g_1,g_2)\in \Lambda_1\times \Lambda_2$, we have 
$$\sigma(\Delta(u_t)\ot 1)=\sum_{g\in \Lambda} (p_g^t v_{g} \otimes v_{e}\otimes v_{g_2})\ot (1_M\ot v_{g_1} \ot v_e),$$
and hence 
$$\Theta(\sigma(\Delta(u_t)\ot 1))=\sum_{g\in \Lambda}\sum_{s\in \Gamma}(p_g^t v_{g} \otimes v_{e}\otimes v_{g_2})\ot (1_M\ot q_s^{(g_1,e)}u_s\ot u_s).$$
The same argument as in Theorem~\ref{thm: rel solid in ME} produces two sequences of projections $\{p_n\}\subset B$ and $\{q_m\}\subset A$
	such that they increase to $1$ strongly and 
	for each $t\in \Gamma$ (resp.\ $g\in \Lambda$), there exists $N_t\in \N$ (resp.\ $M_g\in \N$)
	such that $p_n\sum_{g\in \Lambda} p_g^t=p_n\sum_{g\in F}p_g^t$ (resp.\ $q_m\sum_{s\in \Gamma} q_s^g=q_m\sum_{s\in E}q_s^g$)
	for any $n\geq N_t$ (resp.\ $m\geq M_g$), where $F\subset \Lambda$ and $E\subset \Gamma$ are some finite sets.

Now we view $p_n=p_n\ot 1_{L\Lambda \ovt M}= M\ovt L\Lambda\ot 1_M \subset  \cM\ovt M=\cM_1$
	and consider $\alpha_n(x)=\Ad(p_n)(x)+\tau_{\cM_1}(x)p_n^\perp$ as a u.c.p.\ map on $\cM_1$.
Similarly, we consider $\beta_m^0(x)=\Ad(q_m)(x)+\tau_M(x)q_m^\perp$ for $x\in M$
	and set $\beta_m=\beta_m^0\ot \id_{L\Gamma_2\ovt \B(\ell^2\Gamma_1)}$.
Put $\theta_{n,m}:=\alpha_n\ot \beta_m \in UCP( \cP)$ and enumerate $\Gamma=\{t_i\}_{i\in \N}$.
For each $i\in \N$, set $n_i:=\max_{j\leq i}\{N_{t_i}\}$, $F_i\subset \Lambda$ to be a finite such that
	$p_{n_i}\sum_{g\in \Lambda} p_g^{t_j}=p_{n_i}\sum_{g\in F_i}p_g^{t_j}$ for all $i\leq j$,
	and $m_i:=\max_{g\in F_i }\{M_{(g_1,e)}\}$, where $g=(g_1, g_2)$.
It follows that 
$$
\theta_i(\Theta(\sigma(\Delta(u_{t_j})\ot1)))=\sum_{g\in F_i} \sum_{s\in E}( p_{n_i}( p_g^t v_g) p_n\ot v_e\ot v_{g_2}) \ot (1_M \ot q_{m_i}q_s^{(g_1,e)} u_s q_{m_i} \ot u_s)\in \cM_1\ovt M\otimes_{\rm alg} \C\Gamma
$$ 
for all $j\leq i$,
	where $\theta_i:=\theta_{(n_i, m_i)}$ and $E\subset \Gamma$ is some finite subset.
Lastly, note that $\theta_i$ is $\mathcal A$-bimodular and $\Theta(z)=z\ot 1_{L\Gamma}$ and hence
$$
\theta_i \Big(\Theta\big(\sigma \big(\Delta(u_{t_j})\ot 1\big)z \big)\Big)
	=\theta_i\Big(\Theta \big(\sigma \big(\Delta(u_{t_j}\ot 1\big)\big)\Big) (z\ot1_{L\Gamma})\in \cM_1\ovt M\ot_{\rm alg} \C\Gamma,
$$
	as desired. 
\end{proof}

\section{Proofs of Theorem \ref{theorem.upf.biexact} and Theorem \ref{theorem.product.rigidity}}\label{sec: proof upf}

%The goal of this section is to prove Theorem~\ref{theorem.upf.biexact} and the following theorem.

Throughout this section we assume Notation~\ref{notation} that will be used in the proofs of Theorem \ref{theorem.upf.biexact} and Theorem \ref{theorem.product.rigidity}.

%\begin{note}
%Let $\Lambda$ be a countable i.c.c. group that is measure equivalent to a product $\Gamma=\Gamma_1\times\dots\times\Gamma_n$ of $n\ge 1$ groups. By using \cite[Lemma 3.2]{Fu99A}, there exist $d\ge 1$, free ergodic pmp actions $\Gamma\car (X,\mu)$ and $\Lambda\car (Y,\nu)$ such that 
%$$
%\mathcal R(\Lambda\car Y)=\mathcal R(\Gamma\times \mathbb Z/d\mathbb Z \car X \times \mathbb Z/d\mathbb Z)\cap (Y\times Y).
%$$
%Here, we considered that $\mathbb Z/d\mathbb Z \car (\mathbb Z/d\mathbb Z,c)$ acts by addition and $c$ is the counting measure. We also identified $Y$ as a measurable subset of $X\times \mathbb Z/d\mathbb Z$ and denote $p=1_Y \in L^\infty(X\times \mathbb Z/d\mathbb Z)$. Note that $L^\infty(\mathbb Z/d\mathbb Z)\rtimes \mathbb Z/d\mathbb Z=\mathbb M_d(\mathbb C)$. Hence, by letting $B=L^\infty(Y)$, $A=L^\infty(X)\otimes M_d(\mathbb C)$ and $M=A\rtimes\Gamma$, we have $pMp=B\rtimes\Lambda$ and $B\subset pAp$. Denote by $\{u_g\}_{g\in \Gamma}$ and $\{v_\lambda\}_{\lambda\in \Lambda}$ the canonical unitaries implementing the actions $\Gamma\car A$ and $\Lambda\car B$, respectively.
%
%Following \cite{PoVa10a} we define the $*$-homomorphism $\Delta: pMp\to pMp \bar\otimes L(\Lambda)$ by $\Delta(bv_\lambda)=bv_\lambda\otimes v_\lambda$, for all $b\in B,\lambda\in\Lambda$. One can extend $\Delta$ to a $*$-homomorphism $\Delta:M\to M\bar\otimes L(\Lambda)$ and verify that  $\Delta (M)'\cap M\bar\otimes L(\Lambda)=\mathbb C 1$ since $\Lambda$ is i.c.c. (see the first part of \cite[Section 5]{DHI16} for more details).     
%\end{note}

\begin{lem}\label{lemma.have.fi}
Assume that $\Gamma=\times_{i=1}^n\Gamma_i$ and $\Gamma_1,\dots,\Gamma_n$ are nonamenable biexact.
Let $\Sigma_1,\Sigma_2$ be commuting subgroups of $\Lambda$ satisfying $[\Lambda: \Sigma_1\Sigma_2]<\infty$. 

Then there exist a non-zero projection $e\in B^{\Sigma_1\Sigma_2}$ and disjoint subsets $T_1,T_2\subset \{1,\dots,n\}$ such that $L(\Sigma_i)e$ is amenable relative to $A\rtimes\Gamma_{T_i}$ inside $M$ for any $i\in\{1,2\}$.   

Moreover, note that for any non-zero projection $r\in B$ one can choose $e\in B^{\Sigma_1\Sigma_2}$ with $re\neq 0$.
\end{lem}

\begin{proof}
We only prove the moreover part.
Since $B^{\Sigma_1\Sigma_2}$ is completely atomic, let $e\in B^{\Sigma_1\Sigma_2}$ be a projection such that $B^{\Sigma_1\Sigma_2}e=\mathbb Ce$ and $re \neq 0$.
By \cite[Proposition 2.7]{PoVa14I} for any $i\in\{1,2\}$ there exists a minimal subset $T_i\subset\{1,\dots, n\}$ such that $L(\Sigma_i)e$ is amenable relative to $A\rtimes\Gamma_{T_i}$. Assume by contradiction that there is $j\in T_1\cap T_2$.
By  Theorem \ref{thm: rel solid in ME}, we get that $L(\Sigma_1)e\prec_M A\rtimes\Gamma_{\hat j}$ or $L(\Sigma_2)e$ is amenable relative to $A\rtimes\Gamma_{\hat j}$. By \cite[Proposition 2.7]{PoVa14I} the second possibility contradicts the minimality of $T_2$. Assume now that $L(\Sigma_1)e\prec_M A\rtimes\Gamma_{\hat j}$ holds.
Since $\Lambda$ is i.c.c., we get that $\mathcal N_{pMp}(L(\Sigma_1))'\cap pMp\subset B^{\Sigma_1\Sigma_2}$. Since $B^{\Sigma_1\Sigma_2}e=\mathbb Ce$ we further deduce that $L(\Sigma_1)e\prec_M^s A\rtimes\Gamma_{\hat j}$. This would imply by \cite[Lemma 2.6(3)]{DHI16} that $L(\Sigma_1)e$ is amenable relative to $A\rtimes\Gamma_{\hat j}$, which contradicts the minimality of $T_1$.
\end{proof}

\subsection{Proof of Theorem \ref{theorem.upf.biexact}}
%By applying Lemma \ref{lemma.have.fi} there exist disjoint minimal subsets $S_1$ and $S_2$ of $\{1,\dots,n\}$ such that $L(\Lambda_i)$ is amenable relative to $A\rtimes\Gamma_{S_i}$ inside $M$ for any $1\leq i\leq 2$. 

%The conclusion of the theorem follows by showing that $S_1\cup S_2 = \{1,\dots,n\}$. 
%Indeed, Let $j\in S_2$. By applying Theorem \ref{thm: rel solid in ME} we get that either $L(\Lambda_1)\prec_M A\rtimes \Gamma_{\widehat j}$ or $L(\Lambda_2)$ is amenable relative to $A\rtimes\Gamma_{\widehat j}$. The second option contradicts the minimaility of $S_2$. This shows that $L(\Lambda_1)\prec_M A\rtimes \Gamma_{\widehat j}$ for any $j\in S_2$. Since $S_1\sqcup S_2 = \{1,\dots,n\}$, it follows that $L(\Lambda_1)\prec_M A\rtimes \Gamma_{S_1}$. Similarly, $L(\Lambda_2)\prec_M A\rtimes \Gamma_{S_2}$. By Lemma~\ref{lem: flip intertwining ME}, the conclusion of the theorem follows.

%We continue now with proving that $S_1\cup S_2 = \{1,\dots,n\}$. If the groups $\Gamma_i$ have property (T), then the claim follows directly. For the general case, we apply 

We first define a relative version of the flip automorphism from \cite{IM22} (see also \cite{Dri23}).
Define $\mathcal M=M\bar\otimes L(\Lambda)$  and let $\sigma\in {\rm Aut}(\mathcal M\bar\otimes\mathcal M)$ be defined by $\sigma((m\otimes v_{h_1}\otimes v_{h_2})\otimes(m'\otimes v_{h'_1}\otimes v_{h'_2})) = (m\otimes v_{h'_1}\otimes v_{h_2})\otimes (m'\otimes v_{h_1}\otimes v_{h'_2})$ for all $m,m'\in M$, $h_1,h_1'\in\Lambda_1$ and $h_2,h_2'\in\Lambda_2$.
Define $\Delta:M\to \mathcal M$ by letting $\Delta(a v_h)=a v_h \otimes v_h$, for all $a\in A $ and $h\in\Lambda$.
Note that there are $2n$ commuting subalgebras in $\sigma(\Delta(L(\Gamma))\otimes \Delta(L(\Gamma))) \subset \mathcal M \bar\otimes \mathcal M$ denoted $R_1,\dots , R_{2n}$. More precisely, $R_i = \sigma(\Delta(L(\Gamma_i))\otimes 1)$ if $1\leq i\leq n$ and $R_i = \sigma(1\otimes\Delta(L(\Gamma_i)))$ if $n+1\leq i\leq 2n$. Denote $\tilde p = 1\otimes p$, $\tilde {\mathcal M} = M\bar\otimes M$ and note that $\mathcal M \subset \tilde p \tilde {\mathcal M}\tilde p$.  %For any $1\leq i\leq n$ denote $\tilde \Gamma_{\hat i} = \Gamma_{\hat i} \times \Gamma$ and for any $n+1\leq i\leq n$ denote $\tilde \Gamma_{\hat i} = \Gamma \times\Gamma_{\hat i}$. Set $\mathcal A=A\bar\otimes A$ and remark that there is a natural identification $\tilde {\mathcal M}=\mathcal{A}\rtimes(\tilde\Gamma_{\hat i}\times\Gamma_i)$ for any $1\leq i\leq 2n$. 
%Denote by $\mathcal B$ the algebra of elements in $B\bar\otimes 1\otimes 1\bar\otimes B\subset B\bar\otimes B \bar\otimes B \bar\otimes B$ that are invariant under the diagonal action of $\Lambda$ on the positions $1$ and $4$.
%For all $k_1\neq k_2 \in \{1,2,3,4\}$, we denote by $\mathcal B_{k_1,k_2}$ the algebra of elements in $B\bar\otimes B \bar\otimes B \bar\otimes B$ that are invariant under the diagonal action of $\Lambda$ on the positions $k_1$ and $k_2$. Denote $\mathcal B = \mathcal B_{1,4}\cap \mathcal B_{2,3}$. 
Since $\Lambda$ is i.c.c., note that $\mathcal N_{ {\mathcal M}\bar\otimes \tilde p \tilde {\mathcal M} \tilde p}(R_i)'\cap  {\mathcal M} \bar\otimes \tilde p \tilde {\mathcal M} \tilde p\subset \sigma(\Delta(M)\bar\otimes \Delta(M))'\cap  {\mathcal M} \bar\otimes \tilde p \tilde {\mathcal M} \tilde p =:\mathcal B \subset A\bar\otimes 1 \bar\otimes A\bar\otimes B$, for any $1\leq i\leq 2n$.

By applying Theorem~\ref{lem: flip-rel solid}, we get that for any non-zero projection $z\in \mathcal B$ and for all  $1\leq k\leq 2n$ and $1\leq j \leq n$, we have $R_{\hat k}z\prec_{\mathcal M\bar\otimes \tilde{\mathcal M}} \mathcal M \bar\otimes M\bar\otimes ( A\rtimes\Gamma_{\widehat j})$ or $R_{k}z$ is amenable relative to $\mathcal M \bar\otimes M\bar\otimes  ( A\rtimes\Gamma_{\widehat j})$ inside $\mathcal M\bar\otimes\tilde{\mathcal M}$. Note that it is not possible to find  $1\leq k\leq n$ such that $R_{k}z$ and $R_{k+n}z$ are both amenable relative to $\mathcal M \bar\otimes M\bar\otimes  ( A\rtimes\Gamma_{\widehat j})$ for all $1\leq j \leq n$. Indeed, if there would exist such $k$, then \cite[Proposition 2.7]{PoVa14I} implies that $R_{k}z$ and $R_{k+n}z$ are both amenable relative to $\mathcal M \bar\otimes M \otimes  1$.
By letting $z_0 \in R_{\hat k}'\cap \mathcal M\bar\otimes \mathcal M$ be the support projection of $E_{\mathcal M\bar\otimes\mathcal M}(z)$, it follows by \cite[Lemma 2.4]{Dri23}  that $R_{k}z_0$ and $R_{k+n}z_0$  are both amenable relative to $\mathcal M \otimes 1$ inside $\mathcal M\bar\otimes\mathcal M$. Since $E_{\mathcal M\bar\otimes\mathcal M}(z)\in A\otimes 1\otimes A\otimes 1$, it follows that $\sigma(z_0)=z_0$. By considering the flip automorphism $\sigma$, we derive that $(\Delta(L(\Gamma_k))\otimes 1)z_0$ is amenable relative to $M\bar\otimes L(\Lambda_2) \bar\otimes M\bar\otimes L(\Lambda_1)$, but also amenable relative to $M\bar\otimes L(\Lambda_1) \bar\otimes M\bar\otimes L(\Lambda_2)$ inside $\mathcal M\bar\otimes\mathcal M$. By using once again \cite[Proposition 2.7]{PoVa14I}, we get that $(\Delta(L(\Gamma_k))\otimes 1)z_0$ is amenable relative to $M\bar\otimes 1\bar\otimes M\bar\otimes 1$,
	which implies by \cite[Lemma 2.4]{Dri23} and \cite[Lemma 10.2]{IoPoVa13} that $\Gamma_k$ is amenable, contradiction.
Hence, by \cite[Lemma 2.4]{DHI16} it follows that any non-zero projection $z\in \mathcal B$ and for any  $1\leq k\leq n$, there exist a non-zero projection $z_1\in \mathcal B$ with $z\ge z_1$ and $1\leq j \leq n$ satisfying $R_{\hat k}z_1\prec^s_{\mathcal M\bar\otimes \tilde{\mathcal M}} \mathcal M \bar\otimes M\bar\otimes ( A\rtimes\Gamma_{\widehat j})$ or $R_{\widehat {k+n}}z_1\prec^s_{\mathcal M\bar\otimes \tilde{\mathcal M}} \mathcal M \bar\otimes M\bar\otimes ( A\rtimes\Gamma_{\widehat j})$.  

Using this observation finitely many times, one can construct a function $f: \{1,\dots,n\}\to \{1,\dots,n\}$ and a decreasing sequence of non-zero projections $z_1\ge z_2\ge\dots \ge z_n$ from $\mathcal B$ with the property that $R_{\hat k}z_k\prec^s_{\mathcal M\bar\otimes \tilde{\mathcal M}} \mathcal M \bar\otimes M\bar\otimes ( A\rtimes\Gamma_{\widehat{f(k)}})$ or $R_{\widehat {k+n}}z_k\prec^s_{\mathcal M\bar\otimes \tilde{\mathcal M}} \mathcal M \bar\otimes M\bar\otimes ( A\rtimes\Gamma_{\widehat{f(k)}})$ for any $1\leq k \leq n$. Put $s=z_n$. Next, note that $f$ is bijective. Otherwise, by \cite[Proposition 4.4]{CDAD23B} there is $1\leq \ell \leq n$ such that  $\sigma (\Delta(L(\Gamma))\bar\otimes\Delta(L(\Gamma)))\prec_{\mathcal M\bar\otimes \tilde{\mathcal M}} \mathcal{M} \bar\otimes M\bar\otimes(A\rtimes\Gamma_{\widehat \ell}).$ Since $\sigma_{| (A\otimes 1)\bar\otimes (A\otimes 1) }= {\rm id}_{|(A\otimes 1)\bar\otimes (A\otimes 1)},$ \cite[Lemma 2.3]{BeVa14} implies that $\sigma (\Delta(M)\bar\otimes\Delta(M))\prec_{\mathcal M\bar\otimes \tilde{\mathcal M}} \mathcal{M} \bar\otimes M\bar\otimes(A\rtimes\Gamma_{\widehat \ell}).$ Using the definition of $\sigma$, this further implies that $\Delta(M)\prec_{M\bar\otimes M}M\bar\otimes (A\rtimes\Gamma_{\widehat \ell})$. By \cite[Lemma 10.2]{IoPoVa13} we deduce that $L(\Lambda)\prec_M A\rtimes\Gamma_{\widehat \ell}$, which contradicts the fact that $\Gamma_\ell$ is an infinite group.  This shows that $f$ is bijective.

Next, one can define a partition $\{1,\dots,n\}= T_1\sqcup T_2$ by
\begin{equation*}
 \sigma(\Delta(L(\Gamma))\bar\otimes \Delta(L(\Gamma_{\widehat k})))s\prec^s_{\mathcal M \bar\otimes\tilde{\mathcal M}} \mathcal M\bar\otimes M\bar\otimes(A\rtimes\Gamma_{\widehat {f(k)}}), \text{ for all }k\in T_2, 
\end{equation*}
and
\begin{equation*}
 \sigma(\Delta(L(\Gamma_{\widehat{k}}))\bar\otimes \Delta(L(\Gamma)))s\prec^s_{\mathcal M \bar\otimes\tilde{\mathcal M}} \mathcal M\bar\otimes M\bar\otimes(A\rtimes\Gamma_{\widehat {f(k)}}), \text{ for all }k\in T_1. 
\end{equation*}

By \cite[Lemma 2.8(2)]{DHI16} this further implies that
\begin{equation*}
    \sigma(\Delta(L(\Gamma_{{T_2}}))\bar\otimes \Delta(L(\Gamma_{{T_1}}))) \prec_{\mathcal M \bar\otimes\tilde{\mathcal M}} \mathcal M\bar\otimes M\bar\otimes A
\end{equation*}
Since $\sigma(\Delta(L(\Gamma_{{T_2}}))\bar\otimes \Delta(L(\Gamma_{{T_1}})))\subset \mathcal M\bar\otimes\mathcal M$, we further deduce that
\begin{equation*}
    \sigma(\Delta(L(\Gamma_{{T_2}}))\bar\otimes \Delta(L(\Gamma_{{T_1}}))) \prec_{\mathcal M \bar\otimes{\mathcal M}} \mathcal M\bar\otimes M\bar\otimes 1.
\end{equation*}
By applying the flip autmorphism $\sigma$ to the above intertwining, it follows that $\Delta(L(\Gamma_{{T_i}}))\prec_{\mathcal M} M\bar\otimes L(\Lambda_i)$ for any $1\leq i\leq 2$. 
 This implies that $L(\Gamma_{T_i})\prec_M B\rtimes \Lambda_i$ for any $1\leq i \leq 2$. By Lemma~\ref{lem: flip intertwining ME} and Lemma \ref{lem: ME intertwining} we get the conclusion of the proof.
\hfill$\square$

\begin{rem}
There is an alternative way to end the proof of Theorem \ref{theorem.upf.biexact} by avoiding the use of Lemma \ref{lem: flip intertwining ME}: Let $k\in T_2$. Then by Theorem \ref{thm: rel solid in ME}, we get that $L(\Lambda_1)\prec_M A\rtimes\Gamma_{\widehat k}$ or $L(\Lambda_2)$ is amenable relative to $A\rtimes\Gamma_{\widehat k}$. If the second possibility holds, then by \cite[Lemma 2.11]{Be14}, we get that
$B\rtimes\Lambda_2$ is amenable relative to $A\rtimes\Gamma_{\widehat k}$. Together with $A\rtimes\Gamma_{T_2}$ being amenable relative to $B\rtimes\Lambda_2$, we get by \cite{OzPo10I} that $\Gamma_k$ is amenable, contradiction. 
Hence, $L(\Lambda_1)\prec_M A\rtimes\Gamma_{\widehat k}$ for any $k\in T_2$. This implies that $L(\Lambda_1)\prec_M A\rtimes\Gamma_{T_1},$ and consequently, $B\rtimes\Lambda_1\prec_M A\rtimes\Gamma_{T_1}$. The proof now concludes by using either Lemma \ref{lem: normal subgroup intertwine} or by \cite[Proposition 3.1]{DHI16}.
    
\end{rem}

\subsection{Proof of Theorem~\ref{theorem.product.rigidity}}	
For any $1\leq i\leq n$ since $\Gamma_i$ is biexact,
 Theorem~\ref{cor: rel solid in ME} implies that $\Delta(L(\Gamma_n))$ is amenable relative to $M\bar\otimes (A\rtimes \Gamma_{\hat i})$ or $\Delta(L(\Gamma_{\hat n}))\prec_{M\bar\otimes M} M\bar\otimes (A\rtimes \Gamma_{\hat i})$. Note that if $\Delta(L(\Gamma_n))$ is amenable relative to $M\bar\otimes (A\rtimes \Gamma_{\hat i})$ for any $i$, it follows from \cite[Proposition 2.7]{PoVa14I} that $\Delta(L(\Gamma_n))$ is amenable relative to $M\bar\otimes 1$. This implies by \cite[Lemma 10.2]{IoPoVa13} that $\Gamma_n$ is amenable, contradiction. Therefore, there is $j$ such that $\Delta(L(\Gamma_{\hat n}))\prec_{M\bar\otimes M} M\bar\otimes (A\rtimes \Gamma_{\hat j})$. Since $\Gamma_{\widehat n}$ has property (T), it follows that  $\Delta(L(\Gamma_{\hat n}))\prec_{M\bar\otimes M} M\bar\otimes (A_{\hat j}\rtimes \Gamma_{\hat j})$.
 By using Ioana's ultrapower technique \cite[Theorem 3.1]{Ioa11b} (see also \cite[Theorem 3.3]{CdSS15}, \cite[Lemma 5.6]{KV15} and \cite[Theorem 4.1]{DHI16}), it follows that there is a subgroup $\Sigma<\Lambda$ with non-amenable centralizer $C_\Lambda(\Sigma)$ such that $L(\Gamma_{\widehat n})\prec_M B\rtimes\Sigma$. By applying \cite[Lemma 2.3]{BeVa14}, we deduce that
 $A\rtimes\Gamma_{\widehat n}\prec_M B\rtimes\Sigma$.  By using \cite[Lemma 2.4(2)]{DHI16} there is a non-zero projection $r\in \mathcal N_{pMp}(B\rtimes\Sigma)'\cap pMp \subset B^{\Sigma C_\Lambda(\Sigma)}$ such that 
\begin{equation}\label{i1}
A\rtimes\Gamma_{\widehat n}\prec_M (B\rtimes\Sigma)s \text{ for any non-zero projection } s\in B^{\Sigma }r.
\end{equation}
Next, we claim that $(B\rtimes\Sigma)r \prec_M^s A\rtimes\Gamma_{\hat n}$.
Let $s$ be a non-zero projection in $B^{\Sigma C_\Lambda(\Sigma)}$.
 By Theorem \ref{thm: rel solid in ME} and \cite[Lemma 2.3]{BeVa14}, it follows that for any $i\in\{1,\dots,n-1\}$ we have $(A\rtimes \Sigma)s\prec_M A\rtimes\Gamma_{\widehat i}$ or $L(C_\Lambda(\Sigma))s$ is amenable relative to $A\rtimes\Gamma_{\widehat i}$. The first possibility together with \eqref{i1} and \cite[Lemma 2.4]{Dr19} imply that $A\rtimes\Gamma_{\widehat n}\prec_M A\rtimes\Gamma_{\widehat i}$, which implies that $\Gamma_i$ is not infinite, contradiction. Hence, $L(C_\Lambda(\Sigma))s$ is amenable relative to $A\rtimes\Gamma_{\widehat i}$ for any $i\in\{1,\dots,n-1\}$, which implies by \cite[Proposition 2.7]{PoVa14I} that $L(C_\Lambda(\Sigma))s$ is amenable relative to $A\rtimes\Gamma_{n}$. Since $C_\Lambda(\Sigma)$ is nonamenable, we deduce that $L(C_\Lambda(\Sigma))s$ is not amenable relative to $A\rtimes\Gamma_{\widehat n}$. By using once again By Theorem \ref{thm: rel solid in ME} and \cite[Lemma 2.3]{BeVa14}, we get that $(A\rtimes \Sigma)s\prec_M A\rtimes\Gamma_{\widehat n}$, proving the claim.

By \cite[Proposition 3.1]{DHI16}, $\Sigma$ is measure equivalent to $\Gamma_{\hat n}$, and hence by \cite[Corollary 1.4]{Fu99A}, it follows that $\Sigma$ has property (T). It thus implies that $L(\Sigma)e\prec_M^s M_{\hat n}$. Note that $\Delta:=\{\lambda\in \Lambda \; |\; \mathcal O_\Sigma(\lambda) \text{ is finite}\}$ is normalized by $\Sigma$, where $\mathcal O_\Sigma(\lambda)=\{ \eta \lambda \eta^{-1} \; | \; \eta\in\Sigma \}$. Let $f\in B^{\Delta\Sigma}e$ be a non-zero projection. By passing twice to relative commutants we note that $M_n\prec (B\rtimes\Delta\Sigma)f$, and thus, $B^{\Delta\Sigma}f\prec_M M_{\hat n}$. Also, by passing to relative commutants in \eqref{i1}, we get that $B^{\Delta\Sigma}f\prec_M M_n$. Since $\mathcal N_{pMp}(B^{\Delta\Sigma})'\cap pMp\subset B^{\Delta\Sigma}$ it follows from \cite[Lemma 2.4]{DHI16} that  $B^{\Delta\Sigma}e\prec_M^s M_{\hat n}$ and $B^{\Delta\Sigma}e\prec_M^s M_{n}$. By \cite[Lemma 2.8(2)]{DHI16} it follows that $B^{\Delta\Sigma}e\prec_M^s \mathbb C 1$, which implies that there is a non-zero projection $f_0\in B^{\Delta\Sigma}e$ such that $B^{\Delta\Sigma}f_0 = \mathbb C f_0$. Since $M_n\prec_{M} (B\rtimes\Delta\Sigma)f_0$ and $M_{\hat n}\prec_{M} (B\rtimes\Delta\Sigma)f_0$, it follows from \cite[Lemma 2.6]{Dr19} that $M\prec_{M} B\rtimes \Delta\Sigma$, and hence, $[\Lambda: \Delta\Sigma]<\infty$. 

Next, since $\Lambda$ has property (T), the group $\Delta\Sigma$ has property (T) as well. Note that $\lambda\in \Delta$ if and only if there is a finite index subgroup $\Omega<\Sigma$ with $\lambda\in C_\Lambda(\Omega)$. It thus follows that there is a decreasing sequence $\{\Omega_n\}_{n\ge 1}$ of finite index subgroups of $\Sigma$ for which $\Delta=\cup_{n\ge 1} C_\Lambda(\Omega_n)$. Since $\Delta\Sigma$ has property (T) it therefore follows that there is $n\ge 1$ for which $[\Lambda: \Omega_n C_\Lambda(\Omega_n)]<\infty.$ Denote $\Sigma_1=\Omega_n$ and $\Sigma_2=C_\Lambda(\Omega_n)$.
Since $\Lambda$ is i.c.c., it follows that $\Sigma_1,\Sigma_2$ have property (T). By using Lemma \ref{lemma.have.fi} there exist a non-zero projection $e\in B^{\Sigma_1\Sigma_2}$ and a partition $\{1,\dots,n\}=T_1\sqcup T_2$ such that $L(\Sigma_i)e$ is amenable relative to $A\rtimes\Gamma_{T_i}$ inside $M$ for any $i\in\{1,2\}$. Therefore, by \cite[Lemma 2.6(1)]{DHI16} we get that $L(\Sigma_i)e\prec_M^s A_{T_i}\rtimes\Gamma_{T_i}$ for any $i\in\{1,2\}$. 
Thus, we can apply \cite[Theorem 3.1]{Dr19} and obtain
that there exist commuting subgroups $\Lambda_1,\Lambda_2<\Lambda$ and free ergodic pmp actions $\Lambda_i\car Y_i$ for $i\in\{1,2\}$ such that $\Lambda\car Y$ is induced from $\Lambda_1\times\Lambda_2\car Y_1\times Y_2$ and $\Gamma_{T_i}\car X_{T_i}$  is stably orbit equivalent to $\Lambda_i\car Y_i$. The proof ends by repeating the argument finitely many times.

\begin{rem}
In the above proof, one can show that $T_1$ can be taken to be $\hat n$ and $T_2 = \{n\}$. Indeed, since $A\rtimes\Gamma_{\widehat n}\prec_M (B\rtimes\Sigma)r$ and $[\Sigma:\Sigma_1]$, we get $A\rtimes\Gamma_{\widehat n}\prec_M (B\rtimes\Sigma_1)r$. 
By \cite[Lemma 2.4(4)]{DHI16} there is a non-zero projection $r_1\in B^{\Sigma_1\Sigma_2}r$ for which 
\begin{equation}\label{i2}
A\rtimes\Gamma_{\widehat n}\prec_M (B\rtimes\Sigma_1)s  \text{ for any non-zero projection  } s\in B^{\Sigma_1\Sigma_2 }r_1.    
\end{equation}
By using the moreover part of Lemma \ref{lemma.have.fi} one may assume that $e\in B^{\Sigma_1\Sigma_2}$ satisfies $er_1\neq 0$.
From \eqref{i2} we get $A\rtimes\Gamma_{\widehat n}\prec_M (B\rtimes\Sigma_1)er_1$ and by using \cite[Lemma 3.4]{Va07a}, we deduce $A\rtimes\Gamma_{\widehat n}\prec_M (B\rtimes\Sigma_1)e$. 
By \cite[Lemma 2.3]{BeVa14} we have  $(B\rtimes\Sigma_1)e\prec_M^s A\rtimes\Gamma_{T_1}$, and hence, by using \cite[Lemma 3.7]{Va07a}, we derive $A\rtimes\Gamma_{\widehat n}\prec_M A\rtimes\Gamma_{T_1}$, which shows that $T_1 = \hat n$ and $T_2 = \{n\}$.
\end{rem}

\section{Rigidities for infinite direct sum groups}\label{sec: infinite direct sum}

\begin{proof}[Proof of Theorem~\ref{thm: infinite ME}]
Let $\{\Gamma_i\}_{i\in I}$ and $\{\Lambda_j\}_{j\in J}$ be families of nonamenable biexact groups
	and suppose $\Gamma:=\oplus_{i\in I}\Gamma_i\actson (X,\mu)$ and $\Lambda:=\oplus_{j\in J} \Lambda_j\actson (Y,\nu)$ 
	are stable orbit equivalent free ergodic p.m.p.\ actions.
It follows that we may find $d,\ell\in \N$ such that $Y=Y\times \{0\}\subset X\times \Z/d\Z\subset Y\times \Z/\ell\Z$
	and $$\cR(\Gamma\times \Z/d\Z\actson X\times \Z/d\Z)\cap (Y\times Y)=\cR(\Lambda\actson Y),$$
	$$\cR(\Lambda\times \Z/\ell\Z\actson Y\times Z/\ell \Z)\cap ((X\times \Z/d\Z)\times (X\times \Z/d\Z))
		=\cR(\Gamma\times \Z/d\Z\actson X\times \Z/d\Z).$$

Put $A=L^\infty(X)$, $A_1=L^\infty(X)\ot \M_d(\C)$, $B=L^\infty(Y)$, $B_1=L^\infty(Y)\ot \M_\ell(\C)$,
	$p=1_{A_1}\in B_1$, $q=1_B\in A_1$ and $M=B_1\rtimes\Lambda$.
By Proposition~\ref{prop: infinite rel biexact} and \cite[Proposition 8.3]{DP22}, 
	we have $M$ is biexact relative to $\{B_1\rtimes\Lambda_{\widehat j} \}_{j\in J}$ and hence for each $i\in I$,
	there exists $\sigma(i)\in J$ such that $L\Gamma_{\widehat i}\prec_M B_1\rtimes\Lambda_{\widehat{\sigma(i)}}$.
By Lemma~\ref{lem: ME intertwining}, we have $\Gamma_{\widehat i}\preccurlyeq_\Omega \Lambda_{\widehat {\sigma (i)}}\times \Z/\ell\Z$
	and thus $\Gamma_{\widehat i}\preccurlyeq_{\Omega} \Lambda_{\widehat{\sigma(i)}}$,
	where $\Omega=(X\times \Z/d\Z)\times (\Lambda\times \Z/\ell\Z)=(Y\times \Z/\ell \Z)\times (\Gamma\times \Z/d\Z)$
	is the ME coupling arising from the stable orbit equivalent actions \cite[Theorem 3.3]{Fu99A}.
Note that $\sigma: I\to J$ is injective as otherwise Lemma~\ref{lem: commuting intertwine} shows that $\Lambda_j$ is finite,
	contradicting the nonamenable assumption.

Reversing the roles of $\Gamma_i$'s and $\Lambda_j$'s, we obtain an injection $\rho: J\to I$ such that $L\Lambda_{\widehat j}\prec_{pMp} A_1\rtimes{\Gamma_{\widehat{\rho(j)}}}$.
The same argument as above yields that $\Lambda_{\widehat j}\preccurlyeq_{\Omega_0} \Gamma_{\widehat {\rho(j)}}$,
	where $\Omega_0=(X\times \Z/d\Z)\times \Lambda= Y\times (\Gamma\times \Z/d\Z)$,
	and this in turn implies that $\Lambda_{\widehat j}\preccurlyeq_{\Omega}\Gamma_{\widehat{\rho(j)}}$.

We claim that $\rho\circ \sigma=\id_I$. Otherwise, there would exist some $i\in I$ such that
	$\Gamma_i\preccurlyeq_\Omega \Lambda_{\widehat j}$ and $\Lambda_{\widehat j}\prec_\Omega \Gamma_{\widehat i}$,
	which implies that $\Gamma_i$ would be finite by Lemma~\ref{lem: commuting intertwine}.
Similarly, we have $\sigma\circ \rho=\id_J$.
It then follows from Lemma~\ref{lem: normal subgroup intertwine} that $\Gamma_i\sim_{\rm ME} \Lambda_{\sigma(i)}$ for
	all $i\in I$.
\end{proof}

\begin{rem}
Observe that the conclusion of Theorem~\ref{thm: infinite ME} holds whenever one has $L\Gamma_{\widehat i}\prec_M B_1\rtimes\Lambda_{\widehat{\sigma(i)}}$ 
	and $L\Lambda_{\widehat j}\prec_{pMp} A_1\rtimes{\Gamma_{\widehat{\rho(j)}}}$ as above.
	Thus the same conclusion holds if each $\Gamma_i$ and $\Lambda_j$ are groups with positive first $\ell^2$-Betti number,
	and more generally, in class ${\mathscr M}$ as in \cite{Dri23}.
\end{rem}

\begin{proof}[Proof of Corollary \ref{cor.ME.wreathproduct}]
Let $A\wr\Gamma\car (X,\mu)$ and $B\wr\Lambda\car (Y,\nu)$ be free ergodic pmp actions for which we have the identifications $p(L^\infty(X)\rtimes A\wr\Gamma)p =  L^\infty(Y)\rtimes B\wr\Lambda$ and $L^\infty(X)p=L^\infty(Y)$ for a non-zero projection $p\in L^\infty(X)$. Denote $M=L^\infty(X)\rtimes A\wr\Gamma$.

Note that it is enough to show $L^\infty(X)\rtimes A^{(\Gamma)}\prec_M L^\infty(Y)\rtimes B^{(\Lambda)}$ and $L^\infty(Y)\rtimes B^{(\Lambda)}\prec_M L^\infty(X)\rtimes A^{(\Gamma)}$
by Lemma~\ref{lem: ME intertwining}, Lemma~\ref{lem: normal subgroup intertwine} and Theorem~\ref{thm: infinite ME}. Because of symmetry reasons, we only prove $L^\infty(X)\rtimes A^{(\Gamma)}\prec_M L^\infty(Y)\rtimes B^{(\Lambda)}$. Using \cite[Lemma 2.3]{BeVa14} it is enough to show $L( A^{(\Gamma)})\prec_M L^\infty(Y)\rtimes B^{(\Lambda)}$.

We make the following notation. For any $i\in\Gamma$, denote by $A^i<A^{(\Gamma)}$ the canonical embedding of $A$ on position $i$ and by $\hat i$ the set $\Gamma\setminus\{i\}$. For a subset $I\subset\Gamma$, we denote $A^I=\oplus_{i\in I}A^i$.
By \cite[Theorem 1.4]{PoVa14II} (see also \cite[Lemma 5.2]{KV15}) for any $i\in\Gamma$ we have that (a) $L(A^i)$ is amenable relative to $L^\infty(Y)\rtimes B^{(\Lambda)}$ inside $M$ or (b) $L(A^{\hat i})\prec_M L^\infty(Y)\rtimes B^{(\Lambda)}.$

If there exists $i\in\Gamma$ such that $(b)$ holds, then by \cite[Lemma 2.3]{BeVa14} and \cite[Lemma 2.4]{DHI16} it follows that $(L^\infty(X)\rtimes A^{\hat i})z\prec^s_M L^\infty(Y)\rtimes B^{(\Lambda)}$, where $z\in \mathcal N_{M}(L^\infty(X)\rtimes A^{\hat i}))'\cap M\subset L^\infty(X)^{A^{(\Gamma)}}$ is  a non-zero projection. Since the projections $(u_gzu_g^*)_{g\in\Gamma}$ cannot be mutually disjoint, it follows that there is $g\neq e$ for which $z_0 := z u_gzu_g^*$ is a non-zero projection of $L^\infty(X)^{A^{(\Gamma)}}$. Thus, we get that $(L^\infty(X)\rtimes A^{\hat i})z_0\prec^s_M L^\infty(Y)\rtimes B^{(\Lambda)}$ and $(L^\infty(X)\rtimes A^{\widehat{gi}})z_0\prec^s_M L^\infty(Y)\rtimes B^{(\Lambda)}$. 
By using \cite[Proposition 4.4]{CDAD23B}, it follows that $L(A^{(\Gamma)})\prec_M L^\infty(Y)\rtimes B^{(\Lambda)}$, proving the claim.

Assume now that (a) holds for any $i\in \Gamma$. By \cite[Lemma 2.6]{DHI16} it follows that for any non-zero projection $z\in L^\infty(X)^{A^{(\Gamma)}}$ we have $L(A^i)z$ is amenable relative to $L^\infty(Y)\rtimes B^{(\Lambda)}$ inside $M$. By \cite[Theorem 1.4]{PoVa14II} it follows that either (i) $L(A^{(\Gamma)})z$ is amenable relative to $L^\infty(Y)\rtimes B^{(\Lambda)}$ inside $M$ or (ii) $L(A^i)z \prec_M L^\infty(Y)\rtimes B^{(\Lambda)}$. 

First, suppose that there exists a non-zero projection $z\in L^\infty(X)^{A^{(\Gamma)}}$ for which $L(A^{(\Gamma)})z$ is amenable relative to $L^\infty(Y)\rtimes B^{(\Lambda)}$ inside $M$. By \cite[Lemma 2.11]{Be14} there is a non-zero projection $z_1\in (L^\infty(X)\rtimes A^{(\Gamma)})'\cap M$ such that $L^\infty(X)\rtimes A^{(\Gamma)}z_1$ is amenable relative to $L^\infty(Y)\rtimes B^{(\Lambda)}$ inside $M$. Since $\mathcal N_{M}(L^\infty(X)\rtimes A^{(\Gamma)})'\cap M = \mathbb C1$, it follows from \cite[Lemma 2.6(2)]{DHI16}  that $L^\infty(X)\rtimes A^{(\Gamma)}$ is amenable relative to $L^\infty(Y)\rtimes B^{(\Lambda)}$ inside $M$. By applying \cite[Theorem 1.4]{PoVa14II} once again it follows that either $L^\infty(X)\rtimes A^{(\Gamma)} \prec_ML^\infty(Y)\rtimes B^{(\Lambda)}$, which proves the claim as before, or $M$ is amenable relative to $L^\infty(Y)\rtimes B^{(\Lambda)}$ inside $M$. The last conclusion implies that $\Lambda$ is amenable, contradiction.

Thus, we can assume that (ii) always holds. More precisely, it means that for any $i\in\Gamma$, we have $L(A^i)z \prec_M L^\infty(Y)\rtimes B^{(\Lambda)}$ for all non-zero projections $z\in L^\infty(X)^{A^{(\Gamma)}}$. This implies that $L(A^i) \prec^s_M L^\infty(Y)\rtimes B^{(\Lambda)}$ for any $i\in\Gamma$. By using \cite[Proposition 4.4]{CDAD23B} and \cite[Lemma 2.3]{BeVa14} it follows that $L^\infty(X)\rtimes A^F \prec^s_M L^\infty(Y)\rtimes B^{(\Lambda)}$ for any finite subset $F\subset\Gamma$. By using the fact that relative amenability is closed under inductive limits (see \cite[Proposition 2.7]{DHI16}) it follows that $L^\infty(X)\rtimes A^{(\Gamma)}$ is amenable relative to $L^\infty(Y)\rtimes B^{(\Lambda)}$ inside $M$. By proceeding as in the previous paragraph, we derive that the proof is completed.
\end{proof}

\begin{thm}[cf.\ \cite{CU20}]\label{main.infinite.product.rigidity}
Let $\{\Gamma_n\}_{n\in \N}$ be a family of nonamenable i.c.c.\ groups.
Suppose $L(\oplus_{n\in \N}\Gamma_n)$ is isomorphic to $ L\Lambda$ for some group $\Lambda$.
If each $\Gamma_n$ is biexact and has property (T),
	then there exist an i.c.c.\ amenable group $A$ and 
	a family of i.c.c.\ biexact groups $\{\Lambda_n\}_{n\in\N}$
	such that $\Lambda=(\oplus_{n\in \N} \Lambda_n)\oplus A$,
    and $L\Gamma_n$ is stably isomorphic to $L\Lambda_n$ for each $n\in \N$.
\end{thm}

The proof of Theorem \ref{main.infinite.product.rigidity} will follow as in \cite{CU20} by inductively using the following result.
See also \cite[Remark 3.9]{CU20}.

\begin{thm}
Let $\{M_i\}_{i\in\mathbb N}$  be a family of biexact II$_1$ factors with property (T) and denote $M=\ovt_{i\in\mathbb N} M_i$. Let $\Lambda$ be a countable group such that $M=L(\Lambda)$.

Then there exist a decomposition $\Lambda=\Lambda_1\times \Lambda_{\widehat 1}$, a positive number $t>0$ and a unitary $u\in M$ such that $uM_1^t u^*= L(\Lambda_1)$ and $u M_{\widehat 1}^{1/t} u^* = L(\Lambda_{\widehat 1})$.

\end{thm}

\begin{proof}
Let $\Delta:M\to M\ovt M$ be the $*$-homomorphism given by $\Delta(v_\lambda)=v_\lambda\otimes v_\lambda$, for any $\lambda\in \Lambda$. By using Corollary \ref{corollary.relative.solidity} there exists $k\in\mathbb N$ such that $\Delta(M_{\widehat 1})\prec_{M\ovt M} M\ovt M_{\widehat k}$ or $\Delta(M_{\widehat 1})\prec_{M\ovt M} M_{\widehat k}\ovt M$. By using the symmetry of the operator $\Delta$, we can assume that 
\begin{equation}\label{eq0}
\Delta(M_{\widehat 1})\prec_{M\ovt M} M\ovt M_{\widehat k}.     
\end{equation}

Next, we apply Ioana's ultrapower technique  \cite[Theorem 4.1]{DHI16} (see also \cite[Theorem 3.1]{Ioa11b})
to deduce that there exists a decreasing sequence $\{\Sigma_{i}\}_{i\in\mathbb N}$ of subgroups of $\Lambda$ such that $M_{\widehat 1}\prec_M L(\Sigma_i)$ for any $i \in\mathbb N$ and $M_k\prec_M L(\cup_{i \in \mathbb N}L(C_\Lambda(\Sigma_i)))$. 

Since $M_k$ has property (T), we may find some $j\in \N$ such that $M_k\prec_M L(C_\Lambda(\Sigma_j))$.
Thus there exist two nonamenable commuting subgroups $\Sigma:=\Sigma_j,\Theta:=C_\Lambda(\Sigma_j)$ of $\Lambda$ such that
\begin{equation}\label{eq1}
M_{\widehat 1}\prec_M L(\Sigma) \text{  and  } M_k\prec L (\Theta).   
\end{equation}

By passing to relative commutants in \eqref{eq1} it follows that
\begin{equation}\label{eq2}
L(\Theta)\prec_M M_1 \text{  and  } L(\Sigma)\prec M_{\widehat k}.    
\end{equation}

We continue by showing that $L(\Sigma)\prec^s_M M_{\widehat k}$ and $k=1$. To this end, let $z\in \mathcal N_{M}(L(\Sigma))'\cap M\subset L(\Sigma\Theta)'\cap M$ be a nonzero projection. Since $L(\Theta)z$ is non-amenable, then Corollary \ref{corollary.relative.solidity} implies that there exists $j\in\mathbb N$ such that $L(\Sigma)z\prec_M M_{\widehat{j}}$. We assume by contradiction that $j\neq k$. By Lemma \ref{lemma.trick} it follows that $\Delta(M_{\widehat 1})\prec^s_{M\ovt M} M\ovt M_{\widehat j}$. Together with \eqref{eq0}, it follows from \cite[Lemma 2.8]{DHI16} that $\Delta(M_{\widehat 1})\prec^s_{M\ovt M} M\ovt M_{\widehat {\{j,k\}}}$. By applying again \cite[Theorem 4.1]{DHI16}, there exists a subgroup $\Sigma_0<\Lambda$ such that $M_{\widehat 1}\prec_M L(\Sigma_0)$ and $M_{\{j,k\}}\prec_m L(C_\Lambda(\Sigma_0))$. By passing to relative commutants, it follows that $L(C_\Lambda(\Sigma_0))\prec_M M_1$. By using Lemma \ref{lemma.trick} it follows that $\Delta(M_{\{j,k\}})\prec_{M\ovt M} M\ovt M_{1}$, and hence, $\Delta(M_{\{j,k\}})\prec^s_{M\ovt M} M\ovt M_{1}$ since $M_{\{j,k\}}$ is regular in $M$. By applying the flip automorphism on $M\ovt M$, it follows that $\Delta(M_{\{j,k\}})\prec^s_{M\ovt M} M_1\ovt M$ as well. By applying \cite[Lemma 2.8]{DHI16} it follows that $\Delta(M_{\{j,k\}})\prec^s_{M\ovt M} M_1\ovt M_{1}$. By using Corollary \ref{corollary.relative.solidity} it is not hard to get that $\Delta(M_j)\prec_{M\ovt M} M\otimes 1$, which contradicts the fact that $M_j$ is a II$_1$  factor. Thus, $L(\Sigma)z\prec_M M_{\widehat k}$, which implies by \cite[Lemma 2.4]{DHI16} that $L(\Sigma)\prec_M^s M_{\widehat k}$. 
Together with \eqref{eq2} it follows from \cite[Lemma 3.7]{Va07} that $M_{\widehat 1}\prec_M M_{\widehat k}$. This implies that $k=1$.

To summarize, we proved that $M_{\widehat 1}\prec_M L(\Sigma)$, $L(\Sigma)\prec_M^s M_{\widehat 1}$, $M_1\prec L (\Theta)$ and $L(\Theta)\prec_M M_1$. The conclusion of the theorem follows now by applying \cite[Theorem 2.3]{Dr21} (see also \cite[Theorem 6.1]{DHI16}).
\end{proof}

\bibliographystyle{amsalpha}
\bibliography{ref}

\end{document}